\documentclass[12pt]{amsart}
\usepackage[pdftex,colorlinks,citecolor=blue,urlcolor=blue]{hyperref}

\usepackage{amssymb}
\usepackage{dsfont}
\usepackage{stmaryrd}
\usepackage{mathabx}
\usepackage{mathptmx}
\usepackage{mathtools}
\usepackage{thmtools}
\usepackage{thm-restate}
\usepackage{cleveref}
\usepackage{babel}
\usepackage[all]{xy}
\usepackage{color}
\usepackage{tikz}
\usepackage{todonotes}
\usepackage{paralist}
\numberwithin{equation}{section}
\usepackage[linesnumbered,ruled,vlined]{algorithm2e}
\usepackage{chngcntr}
\usepackage{subcaption}
\usepackage{comment}

\newcommand{\N}{\mathbb{N}}

\newcommand{\Q}{\mathbb{Q}}
\newcommand{\R}{\mathbb{R}}
\newcommand{\Z}{\mathbb{Z}}
\newcommand{\Fc}{\mathcal{F}}

\newcommand{\Lc}{\mathcal{L}}
\newcommand{\Mcc}{\mathcal{M}}

\newcommand{\Tc}{\mathcal{T}}
\newcommand{\Vc}{\mathcal{V}}

\newcommand{\zerob}{{\bf 0}}
\newcommand{\oneb}{{\mathds{1}}}

\newcommand{\Hr}{{{H}}}
\newcommand{\Lct}{{\widetilde{\Lc_d}(\Delta)}}
\newcommand{\Lcd}{{\Lc_d(\Delta)}}
\newcommand{\Ptd}{{\widetilde{P}_{\partial (\sigma_{d+1})}}}

\DeclareMathOperator{\Image}{Im}

\DeclareMathOperator{\rank}{rank}

\DeclareMathOperator{\pnt}{\raise 0.5mm \hbox{\large\bf.}}

\DeclareMathOperator{\conv}{\mathrm{conv}}
\DeclareMathOperator{\vol}{\mathrm{vol}}
\DeclareMathOperator{\nvol}{\mathrm{nvol}}
\newcommand\thankssymb[1]{\textsuperscript{\@{*}}}

\let\phi=\varphi

\newtheorem{theorem}{Theorem}[section]
\newtheorem{lem}[theorem]{Lemma}
\newtheorem{prop}[theorem]{Proposition}
\newtheorem{cor}[theorem]{Corollary}
\newtheorem{con}[theorem]{Conjecture}

\newtheorem{lem-def}[theorem]{Lemma and Definition}
\newtheorem{prop-def}[theorem]{Proposition and Definition}

\theoremstyle{defi}
\newtheorem{defi}[theorem]{Definition} 
\newtheorem{rem}[theorem]{Remark}
\newtheorem{rem-def}[theorem]{Remark and Definition}
\newtheorem{exa}[theorem]{Example}

\newtheorem{question}[theorem]{Question}
\newtheorem{problem}[theorem]{Problem}

\begin{document}
\title{Laplacian polytopes of simplicial complexes}

\author{Martina Juhnke-Kubitzke}
\address{Universit\"at Osnabr\"uck, Institut f\"ur Mathematik, 49069 Osnabr\"uck, Germany}
\email{juhnke-kubitzke@uos.de}

\author{Daniel K\"ohne} 
\address{Universit\"at Osnabr\"uck, Institut f\"ur Mathematik, 49069 Osnabr\"uck, Germany}
\email{dakoehne@uos.de}
\begin{abstract}
Given a (finite) simplicial complex, we define its \emph{$i$-th Laplacian polytope} as the convex hull of the columns of its $i$-th Laplacian matrix. This extends Laplacian simplices of finite simple graphs, as introduced by Braun and Meyer. After studying basic properties of these polytopes, we focus on the $d$-th Laplacian polytope of the boundary of a $(d+1)$-simplex $\partial(\sigma_{d+1})$. If $d$ is odd, then as for graphs, the $d$-th Laplacian polytope turns out to be a $(d+1)$-simplex in this case. If $d$ is even, we show that the $d$-th Laplacian polytope of $\partial(\sigma_{d+1})$ is combinatorially equivalent to a $d$-dimensional cyclic polytope on $d+2$ vertices. Moreover, we provide an explicit regular unimodular triangulation for the $d$-th Laplacian polytope of $\partial(\sigma_{d+1})$. This enables us to to compute the normalized volume and to show that the $h^\ast$-polynomial is real-rooted and unimodal, if $d$ is odd and even, respectively.

\end{abstract}

\maketitle



\section{Introduction}
Over decades, several lattice polytopes arising from graphs have been studied, extensively. Prominent examples include matching po\-ly\-top\-es, cut polytopes, edge polytopes, adjacency polytopes of several types, among which are symmetric edge polytopes (see e.g., \cite{lovaszplummer1986, BarahonaMahjoub1986, HerzogHibiOhsugi2018, HigashitaniJochemkoMichalek2019, OhsugiTsuchiya2021,SEP}). Following this line of research, in 2017, Braun and Meyer \cite{LaplSimplBraun2017} initiated the study of \emph{Laplacian simplices} that are defined as the convex hull of the columns of the classical Laplacian matrix of a simple graph (see also \cite{meyer2018laplacian,LapSimpDigraphs}). Since each simple graph can be seen as a $1$-dimensional simplicial complex and since to each simplicial complex, we can associate Laplacian matrices, defined via their boundary maps in simplicial homology, it is natural to extend the definition of Laplacian sim\-pli\-ces to arbitrary simplicial complexes and their Lapla\-cians. More precisely, given a simplicial complex $\Delta$ (with a fixed ordering of the vertex set) and its $i$-th Laplacian matrix $\Lc_i(\Delta)\coloneqq \partial_{i+1}\partial_{i+1}^{\intercal}+\partial_{i}^{\intercal}\partial_{i}$, we define the \emph{$i$-th Laplacian polytope $P_\Delta^{(i)}$} of $\Delta$ as the convex hull of the columns of $\Lc_i(\Delta)$. Here, $\partial_i$ and $\partial_{i+1}$ denote boundary maps in simplicial homology.

We initiate the study of Laplacian polytopes by establishing first some general combinatorial and geometric properties and then by focusing on a particular case. More precisely, we consider the situation that the underlying simplicial complex $\Delta$ is the boundary of the $(d+1)$-simplex, denoted by $\partial(\sigma_{d+1})$, and that we take its highest Laplacian $\Lc_d(\partial(\sigma_{d+1}))$. For simplicity, we set $\displaystyle{P_{\partial(\sigma_{d+1})}\coloneqq P_{\partial(\sigma_{d+1})}^{(d)}}$. If $d$ is even, it is easily seen, that, as for graphs, $P_{\partial(\sigma_{d+1})}$ is a $(d+1)$-simplex. If $d$ is odd, the situation is more complicated. By deriving a complete facet description of $P_{\partial(\sigma_{d+1})}$ in this case, we are able to show that $P_{\partial(\sigma_{d+1})}$ is combinatorially equivalent to a $d$-dimensional cyclic polytope on $d+2$ vertices. 

It was shown in \cite{LaplSimplBraun2017} that Laplacian simplices have unimodal $h^\ast$-vectors for certain classes of graphs, including trees, odd cycles and complete graphs. Inspired by these results, we study properties of the $h^\ast$-vectors of general Laplacian polytopes. This is further motivated by the general question under which conditions a lattice polytope has a unimodal $h^\ast$-vector. It was conjectured by Hibi and Ohsugi that this is true for reflexive lattice polytopes that have the integer decomposition property (IDP) \cite{HO}, and, recently, Adiprasito, Papadakis, Petrotou and Steinmeyer could confirm this conjecture in the positive \cite{Adiprasito2022}. However, it is still mysterious what happens if the polytope is not reflexive. 
We consider this question for the Laplacian polytope $P_{\partial(\sigma_{d+1})}$ of the boundary of the $(d+1)$-simplex. Even in this seemingly most simple situation, $P_{\Delta}^{(d)}$ turns out to be not reflexive and hence the mentioned results towards unimodality do not apply. However, the following result shows that $P_{\Delta}^{(d)}$ has at least the integer decomposition property.

\begin{restatable}{thmx}{TheoremA}\label{thm: r.u.t of P}
	 $P_{\partial(\sigma_{d+1})}$ has a regular unimodular triangulation for every integer $d\geq 0$.
\end{restatable}
We note that, combined with \cite[Theorem 1.3]{Athanasiadis}, this result implies that the $h^\ast$-vector of $P_{\partial(\sigma_{d+1})}$ is decreasing in its second half which is obviously implied by but weaker than uni\-modality. The main ingredient for \Cref{thm: r.u.t of P} is the so-called \emph{interior polytope} of  $P_{\partial(\sigma_{d+1})}$, that is defined as the convex hull of the interior lattice points of $P_{\partial(\sigma_{d+1})}$. Indeed, this polytope turns out to be reflexive (after translation to the origin) and miraculously,  $P_{\partial(\sigma_{d+1})}$ happens to be the second dilation of it (after translating both polytopes to the origin). Using edgewise subdivisions, we provide an explicit construction of a regular unimodular triangulation for the interior polytope which then extends to such a triangulation of $P_{\partial(\sigma_{d+1})}$ by \cite[Theorem 4.8]{UnimodularTriangulations}. As a byproduct, we can also compute the normalized volume of $P_{\partial(\sigma_{d+1})}$ (see \Cref{cor: normalized volume}). \Cref{thm: r.u.t of P} combined with the results on the interior polytope enables us to show the following statement:
\begin{restatable}{thmx}{TheoremB}\label{thm: unimodality real-rooted}
    \begin{itemize}
		\item[(a)] $h^\ast\left(P_{\partial(\sigma_{d+1})};t\right)$ has only real roots if $d\in\N$ is odd.
		\item[(b)] $h^\ast\left(P_{\partial(\sigma_{d+1})}\right)$ is unimodal with peak in the middle for every $d\in\N$.
	\end{itemize}
\end{restatable}
We note that if $d$ is odd, then the statement in $(b)$ is just an easy consequence of the one in $(a)$. We conjecture $(a)$ to be true also if $d$ is even.

The paper is organized as follows. \Cref{sec: pre} provides necessary background on simplicial complexes, their Laplacian matrices and lattice polytopes. \Cref{sec: LapMat} collects basic properties of the Laplacian matrix $\Lc_d(\partial(\sigma_{d+1}))$ of the boundary of a simplex.  In \Cref{sec: LapPol}, we introduce the $i$-th Laplacian polytope $P^{(i)}_{\Delta}$
of a simplicial complex $\Delta$. Among others, we compute its number of vertices (\Cref{prop: vertices P(k)}), the dimension of $P_{\partial(\sigma_{d+1})}$ (\Cref{thm: dimension}) and show that $P_{\partial(\sigma_{d+1})}$ is always simplicial (\Cref{thm: simplicial}). The goal of \Cref{sec:combinatorial type} is to derive a complete facet description of $P_{\partial(\sigma_{d+1})}$ and to show that it is combinatorially equivalent to a $d$-dimensional cyclic polytope on $d+2$ vertices if $d$ is even (\Cref{thm: complete facet discription} and \Cref{thm: combinatorial type}). \Cref{sec:rut and h} is devoted to the proofs of Theorems \ref{thm: r.u.t of P} and \ref{thm: unimodality real-rooted}, including the construction and study of the interior polytope of $P_{\partial(\sigma_{d+1})}$. Finally, in \Cref{sec: Questions} we state some open problems and possible future directions.

\section{Preliminaries}\label{sec: pre}
In this section, we provide the necessary background on  simplicial complexes, Laplacian matrices and polytopes. For more information on these topics we refer to \cite{Ziegler, Grunbaum, Munkers84, triangulationsDeLoera, UnimodularTriangulations, Goldberg}. Moreover, we assume the reader to have basic knowledge about graphs (see e.g., \cite{Diestel}).

\subsection{Simplicial complexes and Laplacian matrices}
Given a finite set $V$, a \emph{simplicial com\-plex} $\Delta$ on vertex set $V$ is
a collection of subsets of $V$ that is closed under inclusion. Elements of $\Delta$ are called \emph{faces} and inclusion-wise maximal faces are called \emph{facets}. The \emph{dimension} of a face $F$ is $\dim(F)\coloneqq|F|-1$ and we use $F_i(\Delta)$ to denote the set of $i$-dimensional faces of $\Delta$. The dimension of $\Delta$ is defined as $\dim(\Delta)\coloneqq  \max(i~:~ F_i(\Delta)\neq \emptyset)$. If all facets have the same dimension, $\Delta$ is called \emph{pure}. $0$-dimensional and $1$-dimensional faces of $\Delta$ are called \emph{vertices} and \emph{edges}, respectively. The sets of vertices and edges of $\Delta$ induce a graph in a natural way, which we call the \emph{$1$-skeleton} or \emph{graph} of $\Delta$.  
 Given a $(d-1)$-dimensional simplicial complex $\Delta$, its \emph{$f$-vector} $f(\Delta)=(f_{-1}(\Delta),f_0(\Delta),\ldots,f_{d-1}(\Delta))$ is defined by $f_i(\Delta)\coloneqq|\{f\in\Delta~:~\dim(F)=i\}|$ for $-1\le i\le d-1$ and its \emph{$h$-vector} $h(\Delta)=(h_0(\Delta),h_1(\Delta),\ldots,h_d(\Delta))$ by the polynomial identity
  \begin{equation}\label{eqn: h- and f-vector relation}
\sum_{k=0}^{d}h_k(\Delta) t^{d-k}=\sum_{k=0}^{d}f_{k-1}(\Delta) (t-1)^{d-k}.
\end{equation}
 The polynomials $f(\Delta;t)\coloneqq\sum_{i=-1}^{d-1}f_i(\Delta)t^i$ and $h(\Delta;t)\coloneqq\sum_{i=0}^{d}h_i(\Delta)t^i$ are called the  \emph{$f$}- and \emph{$h$-polynomial} of $\Delta$, respectively.

In order to introduce general Laplacian matrices of a simplicial complex $\Delta$, we need to recall basic notions from simplicial homology. For this purpose, let $\Delta$ be a $(d-1)$-dimensional simplicial complex on vertex set $V$ and assume that the vertices are ordered. Without loss of generality, assume $V=[n]=\{1,\ldots,n\}$ endowed with the natural ordering induced by $\N$. We denote by $C_i(\Delta)$ the $\Q$-vector space with basis $\{e_{\sigma}~:~\sigma\in F_i(\Delta)\}$ and set $C_{i}(\Delta)=\{0\}$ for $i\leq -1$ and $i> d-1$. The \emph{$i$-th boundary map} is the linear map $\partial_i\colon C_i(\Delta)\to C_{i-1}(\Delta)$ defined by 
\begin{equation} \label{eqn: Randabbildung}
    \partial_i(e_\sigma) \coloneqq \sum_{k=1}^{i+1}(-1)^{k-1}e_{\sigma\setminus\{j_k\}},
\end{equation}
where $\sigma=\{j_1<\cdots<j_{i+1}\}\in F_i(\Delta)$. By abuse of notation, we will use $\partial_i$ to denote both, the map and its corresponding matrix. The \emph{$i$-th Laplacian matrix of $\Delta$} is defined as $\Lc_i(\Delta)\coloneqq \partial_{i+1}\partial_{i+1}^{\intercal}+\partial_{i}^{\intercal}\partial_{i}$. Note that $\Lc_i(\Delta)$  provides an endomorphism of $C_i(\Delta)$ which depends on the chosen ordering of the vertices. We recall that $\Hr_i(\Delta;\Q)\coloneqq \ker (\partial_i)/\Image (\partial_{i+1})$ is the \emph{$i$-th (simplicial) homology group} of $\Delta$. 

To provide an explicit description of $\Lc_i(\Delta)$, we need some further notation. 
Faces $F,G\in F_i(\Delta)$ are called \emph{lower adjacent} if $F\cap G\in F_{i-1}(\Delta)$. If, additionally, $e_{F\cap G}$ appears with the same sign in $\partial_i(e_F)$ and $\partial_i(e_G)$, we call  $F\cap G$ the \emph{similar common lower simplex} of $F$ and $G$. Otherwise, $F\cap G$ is referred to as the  \emph{dissimilar common lower simplex} of $F$ and $G$. The \emph{upper degree} of $F\in F_i(\Delta)$, denoted $\deg_U(F)$, is the number of $(i+1)$-faces of $\Delta$ containing $F$. We will use the following description of $\Lc_i(\Delta)$ from \cite[Theorem 3.3.4]{Goldberg}: 
\begin{theorem}\label{thm: LapMat Thm}
Let $\Delta$ be a simplicial complex on vertex set $[n]$, ordered $1<\cdots <n$, and let $i\in \N$ with $0\le i\le \dim(\Delta)$. For $F, G\in F_i(\Delta)$, let $\ell_{F,G}$ denote the entry of $\Lc_i(\Delta)$ in row and column corresponding to $F$ and $G$, respectively. Then,  $\Lc_i(\Delta)$ is symmetric. Moreover: 
\begin{itemize}
\item[(i)] If $i=0$, then $\ell_{F,G}=\deg_U(F)$ if $F=G$, $\ell_{F,G}=-1$ if $F\cup G\in F_{i+1}(\Delta)$, and $\ell_{F,G}=0$, otherwise. 
\item[(ii)] If $i>0$, then 
\begin{equation*}
			\ell_{F,G}=\begin{cases}
		\deg_U(F)+i+1, & \mbox{if $F=G$,}\\
			1, & \mbox{if $F\neq G$, $F\cup G\notin F_{i+1}(\Delta)$, $F\cap G\in F_{i-1}(\Delta)$ similar}\\
			-1, & \mbox{if $F\neq G$, $F\cup G\notin F_{i+1}(\Delta)$, $F\cap G\in F_{i-1}(\Delta)$ dissimilar}\\
					0, & \mbox{otherwise.} \\
		\end{cases}
		\end{equation*}
  \end{itemize}
\end{theorem}
Note that if $i=0$ in the previous theorem, then $\Lc_0(\Delta)$ coincides with the classical Laplacian matrix of the graph of $\Delta$ (from graph theory). 

\subsection{(Lattice) polytopes} 

A \emph{polytope} $P$ is the convex hull of finitely many points in $\R^d$. If $\dim P=k$, we call $P$ a \emph{$k$-polytope}. A linear inequality $a^{\intercal}x\le b$ for $a\in \R^d$ and $b\in \R$  is called a \emph{valid inequality} for $P$ if $a^{\intercal}y\leq b$ for  all $y\in P$. A (proper) \emph{face} of $P$ is a (non-empty) set of the form $P\cap\{x\in\R^d~:~a^{\intercal}x=b\}$ for some valid inequality $a^{\intercal}x\le b$ with $a\neq \zerob$. Faces of dimension $0$, $\dim P-2$ and $\dim P-1$  are called \emph{vertices}, \emph{ridges} and \emph{facets}, respectively. We use $\Vc(P)$ and $\Fc(P)$ to denote the set of vertices and facets of $P$, respectively. A valid inequality $a^{\intercal}x\le b$ is \emph{facet-defining} if  $F=P\cap \{x\in\R^d~ :~ a^{\intercal}x= b\}$ for some $F\in\Fc(P)$. The \emph{facet-ridge graph} $G(P)$ of $P$ is  the graph on vertex set $\Fc(P)$ where $\{F,G\}$ is an edge if and only if $F$ and $G$ intersect in a ridge.  If $\Vc(P)\subseteq \Z^d$, $P$ is called a \emph{lattice polytope}.  Two lattice polytopes $P$, $Q\subseteq\R^d$ are \emph{unimodular equivalent}, denoted as $P\cong Q$, if there exist a unimodular matrix $U\in\R^{d\times d}$ and a vector $b\in\Z^d$ such that $U\cdot P+b=Q$. We use $\Delta_d$ to denote the \emph{standard $d$-simplex}, i.e.,  $\Delta_d=\conv \{\{\zerob\}\cup\{\mathbf{e}_i\in\R^d~:~i\in[d]\}\}$, where $\mathbf{e}_1,\ldots,\mathbf{e}_d$ denote the standard unit vectors. A polytope $P$ is \emph{simplicial} if all of its facets are simplices. The \emph{normalized volume} of a $d$-dimensional lattice polytope $P\subseteq \R^d$ is given by $\nvol(P)=d!\cdot\vol(P)$, where $\vol(P)$ denotes the usual Euclidean volume. A lattice $d$-simplex $\Delta$ with normalized volume $1$ is called \emph{unimodular}. In this case, $\Delta\cong \Delta_d$. A polytope $P$ is \emph{reflexive} if $P=\{x\in\R^d~ :~ Ax\le \oneb\}$ for an integral matrix $A$, where $\oneb$ denotes the all ones vector. In this case, $\mathbf{0}$ is the unique interior lattice point of $P$. 

A \emph{triangulation} $\mathcal{T}$ of a lattice $d$-polytope $P$ is a subdivision into lattice simplices of dimension at most $d$. We denote the set of  vertices in $\Tc$ by $\Vc(\Tc)$. A triangulation is \emph{unimodular} if all its simplices are. $\mathcal{T}$ is called \emph{regular} if there exists a height function $\omega_P: \Vc(\Tc)\to\R$ such that $\Tc$ is the projection of the lower envelope of the convex hull of $\{(v,\omega_P(v))~:~v\in \Vc(\Tc)\}\subseteq\R^{d+1}$ to the first $d$ coordinates. We note that every triangulation is in particular a simplicial complex.
 Let $P\subseteq \R^d$ be a lattice $d$-polytope. Ehrhart \cite{Ehrhart} proved that the number of lattice points in the $n$-th dilation of $P$, i.e., $|nP\cap\Z^d|$ is given by a polynomial $E_P(n)$ of degree $d$ in $n$ for all integers $n\geq 0$. The \emph{Ehrhart series} of $P$ is 
 \[
 \sum_{n\ge 0}E_P(n)t^n=\frac{h^{\ast}(P;t)}{(1-t)^{d+1}}=\frac{h^\ast_0(P)+h^\ast_1(P)t+\cdots+h^\ast_s(P)t^s}{(1-t)^{d+1}},
 \]
 where $h^{\ast}(P;t)\in\Z[t]$ is a polynomial of degree at most $d$, called \emph{$h^\ast$-polynomial} of $P$. The vector  $h^\ast(P)=(h^\ast_0(P),\ldots,h^\ast_s(P))$ is called \emph{$h^\ast$-vector} of $P$. We will often omit $P$ from the notation and just write $h^\ast=(h^\ast_0,\ldots,h^\ast_s)$ if $P$ is clear from the context. By \cite[Theorem 2.1]{StanleyLatticePolytopes}, it is well-known that $h^\ast_i(P)$ is non-negative for all $i$. If $P$ admits a unimodular triangulation $\Tc$, then $h^\ast(P)=h(\Tc)$ \cite[Corollary 2.5]{StanleyLatticePolytopes}.
Moreover, if $\Tc$ is  a regular unimodular triangulation of $P$, then \[h^\ast_{\lfloor(d+1)/2\rfloor}(P)\ge \cdots\ge h^\ast_{d-1}(P)\ge h^\ast_d(P)\] 
\cite[Theorem 1.3]{Athanasiadis}. It was shown by Hibi in \cite{HibiReflexive} that a lattice $d$-polytope $P\subseteq\R^d$ is reflexive (up to unimodular equivalence) if and only if $P$ contains a unique interior lattice point, and $h^\ast(P)$ is palindromic, i.e., $h^\ast_i(P)=h^\ast_{d-i}(P)$ for all $0\leq i\leq \lfloor d/2\rfloor$.

\section{Laplacian matrices of boundaries of simplices}\label{sec: LapMat}
In this section we investigate basic properties of the Laplacian matrix of the boundary of a simplex that will be useful for deriving properties of the corresponding Laplacian polytopes in \Cref{sec: LapPol}.

We start with an easy general statement. 
\begin{lem}\label{lem: dim homology group}
Let $\Delta$ be a $d$-dimensional simplicial complex. Then 
\[
\rank\Lc_d(\Delta)=f_d(\Delta)-\dim_{\Q}\Hr_d(\Delta;\Q).
\]
\end{lem}
\begin{proof}
We have the following chain of equalities:
\[
\rank\Lc_d(\Delta)=\rank (\partial_d^\intercal\partial_d)=f_d(\Delta)-
\dim_{\Q}\ker(\partial_d^\intercal\partial_d)=f_d(\Delta)-
\dim_{\Q}\ker (\partial_d),\]
where the last equality follows from the fact that $\ker(\partial_d)=\ker(\partial_d^\intercal\partial_d)$. Since $\dim\Delta=d$, we also have $\Hr_d(\Delta;\Q)=\ker  (\partial_d)$, which shows the claim.
\end{proof}

In the following, we let $\sigma_{d+1}=2^{[d+2]}$ be the \emph{$(d+1)$-simplex} and we use $\partial(\sigma_{d+1})$ to denote its boundary, i.e., $\partial(\sigma_{d+1})=\sigma_{d+1}\setminus\{[d+2]\}$. Let $F_i=[d+2]\setminus \{d+3-i\}$ for $1\leq i\leq d+2$ and order the columns and rows of $\Lc_d(\partial(\sigma_{d+1}))$ according to $F_1,\ldots,F_{d+2}$.  We first provide an explicit description of the $d$-th Laplacian matrix in this case.

\begin{theorem}\label{thm: Lap Mat Thm Boundary}
Let $\Delta=\partial(\sigma_{d+1})$. 
Then,  $\Lc_d(\Delta)\in\Z^{(d+2)\times (d+2)}$, $\Lc_0(\Delta) =\left(\begin{array}{cc}
      0 &0  \\
      0 & 0
  \end{array}\right)$
and, for $d\geq 1$, $1\leq i,j\leq d+2$, we have
\begin{equation*}
  \Lc_d(\Delta)_{ij} =\begin{cases}
                                                        d+1, & \mbox{if } i=j, \\
                                                        (-1)^{i+j-1}, & \mbox{otherwise.}
                                                      \end{cases}
\end{equation*}
\end{theorem}
\begin{proof}
Since $f_{d}(\partial(\sigma_{d+1}))=d+2$, we have $\Lc_d(\Delta)\in\Z^{(d+2)\times (d+2)}$.

Assume $d=0$. As $\partial_0$ is the zero map, the statement is immediate.

Now let $d\geq 1.$ Since $\dim\Delta=d$, it follows that $\deg_U(F)=0$ for any $d$-face $F$ of $\Delta$. Using \Cref{thm: LapMat Thm} this implies that  $\Lc_d(\Delta)_{ii}=d+1$ for all $1\leq i\leq d+2$.

Now, let $i\neq j$. Since $\Lc_d(\Delta)$ is symmetric, we can assume that $i<j$. $F_i$ and $F_j$ have the common lower simplex $F_i\cap F_j=[d+2]\setminus \{d+3-i,d+3-j\}\neq \emptyset$. By \Cref{eqn: Randabbildung}, $e_{F_i\cap F_j}$ appears with sign $(-1)^{d+3-j}$ in $\partial_d(e_{[d+2]\setminus\{d+3-i\}})$ and it appears with sign $(-1)^{d+2-i}$ in $\partial_d(e_{[d+2]\setminus\{d+3-j\}})$. These signs coincide, meaning that $F_i\cap F_j$ is a similar common lower simplex of $F_i$ and $F_j$, if and only if $i+j$ is odd. The claim follows from \Cref{thm: LapMat Thm}.
\end{proof}

The next lemma will be crucial for determining the dimension of the Laplacian polytope of $\partial(\sigma_{d+1})$ in \Cref{thm: dimension}.

\begin{lem}\label{lem: linear combination of columns}
Let $\Delta=\partial(\sigma_{d+1})$. Then, $\Lc_d(\Delta)$ has rank $d+1$ and every $(d+1)$-element subset of the columns (resp. rows) of $\Lc_d(\Delta)$ is linearly independent. 
\end{lem}
\begin{proof}
The first statement follows from \Cref{lem: dim homology group} and the fact that $H_d(\Delta;\Q)=\Q$. Let $1\leq i\leq d+2$. Let $A_i$ be the  $(d+1)\times (d+1)$-matrix obtained from $\Lc_d(\Delta)$ by removing the $i$-th row and column. By definition, $A_i=\Lc_d(\Delta\setminus \{F_i\})$. Since $H_d(\Delta\setminus \{F_i\}),\Q)=0$, this matrix has full rank. As adding any extra row or column to $A_i$ does  not change the rank, the claim follows.

\end{proof}

\begin{lem}\label{lem: Lap Mat with 1 rank}
Let $\Delta=\partial(\sigma_{d+1})$.
Then
 \begin{equation*}
    \rank\begin{pmatrix}
            \Lc_d(\Delta) \\
            1\cdots 1
          \end{pmatrix}= 
          \begin{cases}d+1, &\mbox{if } d \mbox{ is even,}\\
          d+2, &\mbox{if } d \mbox{ is odd.}
          \end{cases}
  \end{equation*}
\end{lem}
\begin{proof}
First assume that $d$ is even. We define $\lambda=(\lambda_1,...,\lambda_{d+2})^\intercal\in\R^{d+2}$ by 
\begin{equation*}\label{eqn: lambda star}
    \lambda_j=\begin{cases}
                  0, & \mbox{if } j \text{\ is odd,} \\
                  \frac{2}{d+2}, & \mbox{if } j \text{\ is even}.
                \end{cases}
  \end{equation*}
  Using \Cref{thm: Lap Mat Thm Boundary} it is straight-forward to verify that $  \Lc_d(\Delta)\cdot \lambda=\oneb$ which, combined with \Cref{lem: linear combination of columns} shows the claim.
  
Now let $d$ be odd and assume by contradiction that $
  \rank\begin{pmatrix}
          \Lc_d(\Delta) \\
          1\cdots 1
        \end{pmatrix}<d+2$. 
\Cref{lem: dim homology group} and \Cref{lem: linear combination of columns} imply that $
  \rank\begin{pmatrix} 
          \Lc_d(\Delta) \\
          1\cdots 1
        \end{pmatrix}=\rank\Lc_d(\Delta)$.
Hence, there exists $\lambda=(\lambda_1,\ldots,\lambda_{d+2})^\intercal\in\R^{d+2}$, such that $\Lc_d(\Delta)\cdot \lambda=\oneb$. Let $\Lc_d(\Delta)_{[d+1]}$ be the matrix obtained from $\Lc_d(\Delta)$ by deleting the last row. Then, we also have $\Lc_d(\Delta)_{[d+1]}\cdot\lambda=\oneb$ and it follows from \Cref{lem: linear combination of columns} that, up to the choice of the last coordinate $\lambda_{d+2}$, the vector  $\lambda$ is unique. Indeed, a direct computation shows that, if $\lambda_{d+2}=\mu$ for some $\mu\in \R$, then we must have
          \begin{equation}\label{eqn: lambda d odd}
            \lambda_j=\begin{cases}
                        \frac{(d+2)\cdot \mu+1}{d+2}, & \mbox{if } j \text{\ is odd,} \vspace{0.5cm} \\
                        -\frac{(d+2)\cdot \mu-1}{d+2}, & \mbox{if } j \text{\ is even}.
                      \end{cases}
          \end{equation}
          However, denoting by $r_{d+2}$ the last row of $\Lc_{d}(\Delta)$, it holds that $r_{d+2}\cdot \lambda = 0 \neq 1$, which yields a contradiction.
\end{proof}

\section{General properties of Laplacian polytopes} \label{sec: LapPol}

The goal of this section is to generalize Laplacian simplices ~--~ as introduced and studied in  \cite{LaplSimplBraun2017, meyer2018laplacian}~--~ that are associated to a graph to arbitrary simplicial complexes and their Laplacian matrices. After stating some basic general properties of what we call Laplacian polytopes, we focus on boundaries of simplices and their highest Laplacians.
 
In the following, given a matrix $M$, we use $\conv(M)$ to denote the polytope given by the convex hull of the columns of $M$.

\begin{defi}\label{def: Laplacian Polytope}
Let $\Delta$ be a $d$-dimensional simplicial complex on $[n]$, ordered $1<\cdots <n$, and let $0\leq k\leq d$. The \emph{$k$-th Laplacian polytope} of $\Delta$ is defined as the  convex hull of the columns of $\Lc_k(\Delta)$, i.e.,
\[P^{(k)}_{\Delta}\coloneq\conv(\Lc_k(\Delta))\subseteq \R^{f_k(\Delta)}.\]
\end{defi}

We want to remark that the $0$-th Laplacian polytope of a simplicial complex coincides with the Laplacian simplex of its $1$-skeleton, as defined in \cite{LaplSimplBraun2017}. The next example shows that different orderings of the vertex set of $\Delta$ may result in polytopes of different dimensions.

\begin{exa}\label{ex: ordering}
        Let $G$ be the 4-cycle on $[4]$ with $E(G)=\{12,23,34,14\}$. If the vertices of $G$ are ordered $1<2<3<4$, then $P^{(1)}_G$ is a 3-simplex. If the vertices of $G$ are ordered $1<2<4<3$, then $P_G^{(1)}$ is a 2-dimensional rectangle.
\end{exa}
\begin{exa}\label{exa: Laplace Polytope}
 $P^{(2)}_{\partial(\sigma_3)}$ is given by the convex hull of the columns of the following matrix:
 \begin{eqnarray*}
  \Lc_2(\partial (\sigma_3))&=&\left(\begin{array}{cccc}
                                   3 & 1 & -1 & 1 \\
                                   1 & 3 & 1 & -1 \\
                                   -1 & 1 & 3 & 1 \\
                                   1 & -1 & 1 & 3
                                 \end{array}\right).
\end{eqnarray*}    
 It will follow from \Cref{lem: unimodular equivalence} that $P^{(2)}_{\partial(\sigma_3)}$ is unimodular equivalent to the square in $\R^2$ with vertices $(1,-1),(-1,1),(3,1)$ and $(1,3)$.
\end{exa}
We start by showing that every column of $\Lc_k(\Delta)$ yields a vertex of $P^{(k)}_{\Delta}$. 
\begin{prop}\label{prop: vertices P(k)}
  Let $\Delta$ be a $d$-dimensional simplicial complex and $0\leq k\leq d$ an integer. Then, $P^{(k)}_{\Delta}$ has $f_k(\Delta)$ many vertices.
\end{prop}
\begin{proof}
Set $m\coloneqq f_k(\Delta)$ and let $v^{(i)}$ denote the $i$-th column of $\Lc_k(\Delta)$. We assume by contradiction that there exists $1\leq i\leq m$, a set $S\subseteq [ m]\setminus\{i\}$ and $\lambda_j\in \R$ with $\lambda_j>0$ and $\sum_{j\in S}\lambda_j=1$ such that $v^{(i)}=\sum_{j\in S}\lambda_j v^{(j)}$. Setting $\lambda_j=0$ if $j\notin S\cup\{i\}$ and $\lambda_i=-1$, we see that $\lambda\coloneqq (\lambda_1,\ldots,\lambda_m)^\intercal\in \ker(\Lc_k(\Delta))$ and hence $\lambda\in \ker(\partial_k)$ by \cite[Corollary 1.3.1]{Mulas}. Let $w^{(\ell)}$ denote the $\ell$-th column of $\partial_k$. 
If $w^{(i)}_\ell=1$, then, since $w^{(j)}_\ell\in\{-1,0,1\}$, $\lambda_j>0$ and $\sum_{j\in S}\lambda_j=1$, we must have $w^{(j)}_\ell=1$ for all $j\in S$. By the same reasoning, it follows that $w^{(j)}_\ell=-1$ for all $j\in S$ if $w^{(i)}_\ell=-1$. As all columns of $\partial_k$ have the same number of non-zero entries, we conclude $w^{(i)}=w^{(\ell)}$ for all $\ell \in S$, which is a contradiction.
\end{proof}

The next proposition gives a sufficient criterion for $P^{(\dim\Delta)}_{\Delta}$ being a simplex. 
\begin{prop}
    \label{prop: L full rank P fd-1-simplex}
Let $\Delta$ be a $d$-dimensional simplicial complex. If $\Hr_d(\Delta;\Q)=0$, then $P^{(\dim\Delta)}_{\Delta}$ is an $(f_d(\Delta)-1)$-simplex.
\end{prop}
\begin{proof}
\Cref{lem: dim homology group} implies that  $\rank\Lc_d(\Delta)=f_d(\Delta)$. Consequently, the columns of $\Lc_d(\Delta)$ are linearly independent which shows the claim.
\end{proof}

 In the following, we focus on the $d$-th Laplacian polytope of  $\partial(\sigma_{d+1})$. To simplify notation, we set $P_{\partial(\sigma_{d+1})}=P_{\partial(\sigma_{d+1})}^{(d)}$. We use $s^{(i)}$ to denote the $i$-th column of $\Lc_d(\partial(\sigma_{d+1}))$. Moreover, given a subset $S\subseteq [d+2]$, we denote by $\Lc_d(S)$ the matrix obtained from $\Lc_d(\partial(\sigma_{d+1}))$ by deleting the rows with indices in $S$.

 Combining \Cref{lem: Lap Mat with 1 rank} and  \cite[p. 4]{Grunbaum} the following formula for the dimension of $P_{\partial(\sigma_{d+1})}$ is immediate.
\begin{lem}\label{thm: dimension}
Let $\Delta=\partial(\sigma_{d+1})$. Then,
 \begin{equation*}
   \dim P_{\Delta}=\begin{cases}
                     d,   & \mbox{if } d \text{\ is even,} \\
                     d+1, & \mbox{if } d \text{\ is odd}.
                   \end{cases}
 \end{equation*}
\end{lem}

The previous statement together with \Cref{prop: vertices P(k)} allows us to conclude:
\begin{cor}\label{thm: P simplex d odd and number of vertices}
Let $d\in \N$ with $d\geq 1$ and $\Delta=\partial(\sigma_{d+1})$. Then, $P_{\Delta}$ has $d+2$ vertices. In particular, $P_{\Delta}$ is a $(d+1)$-simplex, if $d$ is odd.
\end{cor}

\Cref{thm: P simplex d odd and number of vertices} trivially implies that $P_{\partial(\sigma_{d+1})}$ is a simplicial polytope if $d$ is odd. The same statement turns out to be true for $d$ even.
\begin{theorem}\label{thm: simplicial}
	$P_{\partial(\sigma_{d+1})}$ is simplicial for every $d\in\N$.
\end{theorem}
\begin{proof}
Let $\Delta=\partial(\sigma_{d+1})$. If $d$ is odd, then the claim is trivially true by  \Cref{thm: P simplex d odd and number of vertices}.

Now, let $d$ be even. If $d=0$, then $P_\Delta$ is just the origin and as such simplicial. Let $d\geq 2$ and let $F$ be the vertices of a facet of $P_\Delta$. Combining \Cref{thm: dimension} and \Cref{thm: P simplex d odd and number of vertices} it follows that $d\leq |F|\leq d+1$. If, by contradiction, $|F|=d+1$, then \Cref{lem: linear combination of columns} implies that the convex hull of $F$ is $d$-dimensional,  i.e., $F$ cannot be a facet. Consequently, $F$ is a simplex, which finishes the proof.
\end{proof}

 As, by \Cref{thm: dimension}, the Laplacian polytope of $\partial(\sigma_{d+1})$ is never full-dimensional, our next goal is to construct  a polytope that is unimodular equivalent to $P_{\partial(\sigma_{d+1})}$ and full-dimensional with respect to its ambient space. We first need to introduce some further notation. 
 
 We let $\oneb_{\mathrm{even}}$ and $\oneb_{\mathrm{odd}}$ denote the $0-1$-vectors in $\R^{d+2}$ whose even and odd entries are equal to $1$, respectively. Given these definitions, we can easily compute the affine hull of $P_{\partial(\sigma_{d+1})}$.
 
\begin{lem}\label{lem: affine hull}
Let $d\in \N$ with $d\geq 1$ and $\Delta=\partial(\sigma_{d+1})$. 
    \begin{equation*}
    \mathrm{aff}(P_{\Delta})=\begin{cases}
      \left\{x\in\R^{d+2}~:~(\oneb_{\mathrm{odd}}-\oneb_{\mathrm{even}})^\intercal\cdot x=0\right\}, & \mbox{if } d \mbox{ is odd,}\\
      \left\{x\in\R^{d+2}~:~\oneb_{\mathrm{odd}}^\intercal \cdot x=\oneb_{\mathrm{even}}^\intercal\cdot x=\frac{d+2}{2}\right\}, & \mbox{if } d \mbox{ is even}.\\
     \end{cases}
  \end{equation*}
\end{lem}
\begin{proof} 
By \Cref{thm: dimension}, it is enough to show that all vertices of $P_\Delta$ lie in the specified subspaces of dimension  $d+1$ and $d$, respectively. This can  be seen by a direct computation.
 \end{proof}

The next lemma gives the desired unimodular equivalent polytopes.

\begin{lem}\label{lem: unimodular equivalence}
Let $d\in \N$. The polytope $P_{\partial(\sigma_{d+1})}$ is unimodular equivalent to $\conv(\Lc_d(\{1\}))$ and $\conv\left(\Lc_d(\{1,2\})\right)$ if $d$ is odd and even, respectively.
\end{lem}
\begin{proof}
Define matrices $A,B\in\Z^{(d+2)\times (d+2)}$ as follows
\begin{equation*}
  A=\left(\begin{array}{clccc}
               & & &\oneb_{\mathrm{odd}}^\intercal-\oneb_{\mathrm{even}}^\intercal &  \\ \hline
              0 &\vline\hfill & &  &  \\
              \vdots &\vline\hfill & & E_{d+1} &  \\
              0 &\vline\hfill & &  &
            \end{array}\right)
\quad
\mbox{and} \quad
  B=\left(\begin{array}{cclccc}
               &  & & \oneb_{\mathrm{odd}}^\intercal &  &  \\
               &  & & \oneb_{\mathrm{even}}^\intercal &  &  \\ \hline
              0 & 0 &\vline\hfill&  &  &  \\
              \vdots & \vdots &\vline\hfill& & E_d  &  \\
              0 & 0 &\vline\hfill & &  &
            \end{array}\right),
\end{equation*}
where $E_d$ and $E_{d+1}$ denote identity matrices. Note that $A$ and $B$ are unimodular. By \Cref{lem: affine hull}, we conclude that 
\[
A\cdot P_{\partial(\sigma_{d+1})}=\{0\}\times \conv(\Lc_d(\{1\})),
\]
if $d$ is odd
and
\[
B\cdot P_{\partial(\sigma_{d+1})}=\{((d+2)/2,(d+2)/2)\}\times \conv\left(\Lc_d(\{1,2\})\right),
\]
if $d$ is even. This finishes the proof.
\end{proof}
In the following, we use $\widetilde{P}_{\partial(\sigma_{d+1})}$ to denote the unimodular equivalent polytope to $P_{\partial(\sigma_{d+1})}$ as constructed in \Cref{lem: unimodular equivalence}. By abuse of notation, we will also refer to $\widetilde{P}_{\partial(\sigma_{d+1})}$ as $d$-th Laplacian polytope of $\partial(\sigma_{d+1})$. We also want to remark that, if $d$ is odd, we have the following, easy-to-show containment relation: $\widetilde{P}_{\partial(\sigma_{d+1})}\subseteq \widetilde{P}_{\partial(\sigma_{d+2})}$.


\section{The facet description and the combinatorial type of $P_{\partial(\sigma_{d+1})}$}\label{sec:combinatorial type}
While, for odd $d$, we have already seen that $\widetilde{P}_{\partial (\sigma_{d+1})}$ is a simplex, the goal of this section is to determine the combinatorial type of $P_{\partial(\sigma_{d+1})}$ if $d$ is even. To reach this goal, we will first provide a complete irredundant facet description of $P_{\partial(\sigma_{d+1})}$.

We fix some notation. Let $b^{(\ell)}$ denote the vertex of $\widetilde{P}_{\partial(\sigma_{d+1})}$, that is given by the $\ell$-th column of $\Lc_d(\{1,2\})$. By \Cref{thm: Lap Mat Thm Boundary}, we have $b^{(\ell)}_k=d+1$ if $k=\ell-2$ and  $b^{(\ell)}_k=(-1)^{k+\ell-1}$, otherwise.  

\begin{prop}\label{thm: Facet inequalities and vertices}
	Let $d\geq2$ be even. Then the following inequalities are facet-defining and irredundant for $\widetilde{P}_{\partial (\sigma_{d+1})}$
	\begin{itemize}
	\item[\emph{(i)}] $\oneb^{\intercal}\cdot x\leq d+2$,
	\item[\emph{(ii)}]$\oneb_{\mathrm{odd}}^{\intercal} \cdot x-x_i\leq \frac{d+2}{2}$, where $i\in [d]$ is even,
	\item[\emph{(iii)}] $\oneb_{\mathrm{even}}^{\intercal} \cdot x-x_j\leq \frac{d+2}{2}$, where $j\in [d]$ is odd,
	\item[\emph{(iv)}] $x_i+x_{j}\geq 0$, where $1\leq i<j\leq d$ such that $i+j$ is odd.
\end{itemize}
Moreover, the vertices, that attain equality  in \emph{(i)}--\emph{(iv)}, are given by the sets $\{b^{(\ell)}~:~3\leq \ell\leq d+2\}$, $\{b^{(\ell)}~:~\ell\in [d+2]\setminus \{1,i+2\}\}$, $\{b^{(\ell)}~:~\ell\in [d+2]\setminus \{2,j+2\}\}$ and $\{b^{(\ell)}~:~\ell\in [d+2]\setminus \{i+2,j+2\}\}$, respectively.
\end{prop}

\begin{proof}
We first consider the inequality in (i). If $\ell\in \{1,2\}$, then $b^{(\ell)}\in \{-1,1\}^d$ with alternating entries and hence $\oneb^{\intercal}\cdot b^{(\ell)}<d+2$. Let $3\leq \ell\leq d+2$. As $d$ is even, it follows from above that $b^{(\ell)}$ has one entry equal to $d+1$, $\frac{d}{2}$ entries equal to $1$ and $\frac{d}{2}-1$ entries  equal to $-1$. This implies $\oneb^{\intercal}\cdot b^{(\ell)}=d+2$. Hence, the inequality in (i) defines a facet, whose vertices are given by $\{b^{(\ell)}~:~3\leq \ell\leq d+2\}$, where we use that the affine hull of the latter set is $(d-1)$-dimensional by  \Cref{lem: linear combination of columns}.

Similarly, it is straightforward to verify that the inequalities in (ii)--(iv) are valid for $\widetilde{P}_{\partial (\sigma_{d+1})}$ and that the given sets of vertices are the ones attaining equality. As those all differ and their affine hulls all have dimension $d-1$, it follows that the inequalities are irredundant.
\end{proof}

For the sake of completeness we add the description of the facets of $\widetilde{P}_{\partial (\sigma_{d+1})}$ for $d$ odd.

 \begin{rem} 
 If $d\geq 3$ is odd, using \Cref{thm: Lap Mat Thm Boundary} it is not hard to see, that the following inequalities are facet-defining for $\widetilde{P}_{\partial (\sigma_{d+1})}$
 	\begin{itemize}
	\item[\emph{(i)}] $\oneb^{\intercal}\cdot x\leq d+2$,
	\item[\emph{(ii)}]$2\cdot \oneb_{\mathrm{odd}}^{\intercal} \cdot x-x_i\leq \frac{d+2}{2}$, where $i\in [d+1]$ is even,
	\item[\emph{(iii)}] $2\cdot \oneb_{\mathrm{odd}}^{\intercal} \cdot x+x_j\leq d+2$, where $j\in [d]$ is odd.
\end{itemize}
It is easy to verify that these inequalities are irredundant and as, by \Cref{thm: dimension}, $\widetilde{P}_{\partial (\sigma_{d+1})}$ is a simplex, they provide the complete facet description of $\widetilde{P}_{\partial (\sigma_{d+1})}$. We omit an explicit proof since this description will not be needed. 
 \end{rem}

We state the first main result of this section.

\begin{theorem}\label{thm: complete facet discription}
	For $d$ even, $\widetilde{P}_{\partial (\sigma_{d+1})}$ is completely described by the inequalities in  \Cref{thm: Facet inequalities and vertices}. Moreover, this description is irredundant. In particular, $\widetilde{P}_{\partial (\sigma_{d+1})}$ has $\frac{(d+2)^2}{4}$ many facets.
\end{theorem}

\begin{proof}
We let $\widetilde{\Fc}$ denote the set of facets of $\widetilde{P}_{\partial (\sigma_{d+1})}$, provided by \Cref{thm: Facet inequalities and vertices} and we write $G_{\widetilde{\Fc}}$ for the  subgraph of the facet-ridge graph of $\widetilde{P}_{\partial (\sigma_{d+1})}$ that is induced on vertex set $\widetilde{\Fc}$. It follows from \Cref{thm: simplicial}, that the facet-ridge graph of $\widetilde{P}_{\partial (\sigma_{d+1})}$ is $d$-regular and connected. Since any $d$-regular subgraph does not have a proper $d$-regular subgraph, for the first statement, it suffices to show that $G_{\widetilde{\Fc}}$ is $d$-regular.

 Since $G_{\widetilde{\Fc}}$ is a subgraph of $G(\widetilde{P}_{\partial(\sigma_{d+1})})$,  its maximal degree is at most $d$. Hence, to show the claim, it suffices to show that $|E(G_{\widetilde{\Fc}})|=\frac{d\cdot |V(G_{\widetilde{\Fc}})|}{2}$.

We first count the vertices of $G$. Using \Cref{thm: Facet inequalities and vertices}, we get that
\begin{equation}\label{eq: number facets}
|V(G_{\widetilde{\Fc}})|=1+\frac{d}{2}+\frac{d}{2}+\left(\frac{d}{2}\right)^2=\frac{(d+2)^2}{4},
\end{equation}
Here, the last term in the middle comes from the fact that the inequalities in (iv) are indexed by sets $\{i,j\}$ where $i\in\{2\ell~:´~\ell\in [\frac{d}{2}]\}$ and $j\in\{2\ell-1~:´~\ell\in [\frac{d}{2}]\}$.

It remains to count the number of edges of $G_{\widetilde{\Fc}}$. In the following, we identify a facet in $\widetilde{\Fc}$ with its set of vertices. Given this, we use the following short hand notation for the different types of facets in $\widetilde{\Fc}$. 
\begin{itemize}
\item[(i)] $F=\{b^{(\ell)}~:~3\leq\ell\leq d+2\}$;
\item[(ii)] $E_i=\{b^{(\ell)}~:~\ell\in [d+2]\setminus \{1,i+2\}\}$, where $i\in [d]$ is even;
\item[(iii)]$O_j=\{b^{(\ell)}~:~\ell\in [d+2]\setminus \{2,j+2\}\}$, where $j\in [d]$ is odd;
\item[(iv)] $F_{k,m}=\{b^{(\ell)}~:~\ell\in [d+2]\setminus \{k+2,m+2\}\}$ for $1\leq k<m\leq d$ such that $k+m$ is odd.
\end{itemize}
We immediately get that
\begin{itemize}
\item[(a)] $|F\cap E_i|=|F\cap O_j|=d-1$ for all even $i\in [d]$ and all odd $j\in [d]$;
\item[(b)] $|F\cap F_{k,m}|=d-2$ for all $1\leq k<m\leq d$;
\item[(c)] $|E_i\cap E_j|=d-1$ for all odd $i,j\in [d]$ with $i\neq j$;
\item[(d)] $|E_i\cap O_j|=d-2$ for all even $i\in [d]$ and all odd $j\in [d]$;
\item[(e)] $|E_i\cap F_{k,m}|=d-1$ iff $i\in \{k,m\}$, $i$ even, $k+m$ odd, and $|E_i\cap F_{k,m}|=d-2$, otherwise;
\item[(f)] $|O_i\cap O_j|=d-1$ for all even $i,j\in [d]$ with $i\neq j$;
\item[(g)] $|O_j\cap F_{k,m}|=d-1$ iff $j\in\{k,m\}$, $j$ odd, $k+m$ odd, and $|O_j\cap F_{k,m}|=d-2$, otherwise.
\item[(h)] $|F_{i,j}\cap F_{k,m}|=d-1$ iff $|\{i,j,k,m\}|=3$ and  $|F_{i,j}\cap F_{k,m}|=d-2$, otherwise.
\end{itemize}
Since edges of $G_{\widetilde{\Fc}}$ are given by tuples of facets intersecting in $d-1$ vertices, we get $d$ edges in (a), $0$ edges in (b) and  (d), $\binom{d/2}{2}$ edges in each of (c) and (f), $\left(\frac{d}{2}\right)^2$ edges in each of (e) and (g) and $2\cdot\frac{d}{2}\cdot\binom{d/2}{2}$ edges in (h). This yields
\[
|E(G_{\widetilde{\Fc}})|=d+2\cdot \binom{d/2}{2} +2\cdot \left(\frac{d}{2}\right)^2+d\cdot \binom{d/2}{2}=\frac{d(d+2)^2}{8}=\frac{d\cdot |V(G_{\widetilde{\Fc}})|}{2}.
\]
It follows that $G_{\widetilde{\Fc}}$ is $d$-regular. The \emph{In particular}-statement follows from \eqref{eq: number facets}.
 \end{proof}
 
The previous theorem allows us to determine the combinatorial type of $\widetilde{P}_{\partial (\sigma_{d+1})}$ if $d$ is even. It is well-known (see e.g., \cite[Section 6.1]{Grunbaum}) that there are only finitely many combinatorial types of simplicial $d$-polytopes with $d+2$ vertices. More precisely, any simplicial $d$-polytope with $d+2$ vertices is obtained as the convex hull of a $d$-simplex $T^d$ and a vertex $v$ that is beyond $k$ facets of $T^d$, where $1\leq k\leq d-1$. It is easily seen that the combinatorial type of such a polytope only depends on $k$. Following Gr\"unbaum, we use $T^d_k$ to denote the corresponding combinatorial type. 
Given that we know the number of facets of $\widetilde{P}_{\partial (\sigma_{d+1})}$ (see \Cref{thm: complete facet discription}), we can immediately determine its combinatorial type.

\begin{theorem}\label{thm: combinatorial type}
Let $d$ be even. Then $\widetilde{P}_{\partial (\sigma_{d+1})}$ is of combinatorial type $T^d_{\frac{d}{2}}$. In particular,  $\widetilde{P}_{\partial (\sigma_{d+1})}$ is combinatorially equivalent to a $d$-dimensional cyclic polytope on $d+2$ vertices.
\end{theorem}

\begin{proof}
By \Cref{thm: complete facet discription}, $\widetilde{P}_{\partial (\sigma_{d+1})}$ has $\frac{(d+2)^2}{4}$ facets. Using \cite[Section 6.1, Theorem 2]{Grunbaum}, it follows that this number has to be equal to 
\[
\binom{d+2}{2}-\binom{k+1}{2}-\binom{d+1-k}{2},
\]
where $\widetilde{P}_{\partial (\sigma_{d+1})}$ is of combinatorial type $T^d_k$. Solving for $k$ yields $k=\frac{d}{2}$. The second statement follows from \cite[Section 6.1, Theorem 1]{Grunbaum}.
\end{proof}

We remark that from the previous theorem, we also get a precise formula for the $f$- and $h$-vector of $\widetilde{P}_{\partial (\sigma_{d+1})}$ (see, e.g., \cite{Grunbaum}).

\begin{rem}
Given the precise description of the facets from the proof of \Cref{thm: complete facet discription}, it is not hard to write down a shelling order for $\widetilde{P}_{\partial (\sigma_{d+1})}$ ($d$ even). Namely, one particular shelling is given by
\[
F,E_2,E_4,\ldots,E_d, O_1,O_3,\ldots, O_{d-1},F_{1,2},F_{1,4},\ldots ,F_{1,d},F_{2,3},F_{2,5},\ldots,F_{2,d-1},\ldots, F_{d-1,d}.
\]
\end{rem}


\section{Regular unimodular triangulations and $h^\ast$-vectors}\label{sec:rut and h}
This section is divided into two parts. The goal of the first is to prove \Cref{thm: r.u.t of P}, namely,  that  $\widetilde{P}_{\partial (\sigma_{d+1})}$ admits a regular unimodular triangulation . As a byproduct we will also be able to compute the normalized volume $\widetilde{P}_{\partial (\sigma_{d+1})}$. In the second part, we provide the proof of \Cref{thm: unimodality real-rooted}.

\subsection{Triangulations through interior polytopes}

If $d$ is even, one of our main tools towards the formulated goal is the so-called \emph{interior polytope} $Q_{\partial (\sigma_{d+1})}$ of  $\widetilde{P}_{\partial (\sigma_{d+1})}$, defined as follows:
\[
Q_{\partial (\sigma_{d+1})}=\mathrm{conv}\left(\widetilde{P}_{\partial (\sigma_{d+1})}\setminus \partial\left(\widetilde{P}_{\partial (\sigma_{d+1})}\right) \cap\Z^d\right).
\]

\Cref{fig: Interior2} depicts $\widetilde{P}_{\partial (\sigma_3)}$ and its interior polytope $Q_{\partial (\sigma_{3})}$, both translated to the origin. 
 \begin{figure}[h]
    \centering
    \includegraphics[width=6.0cm]{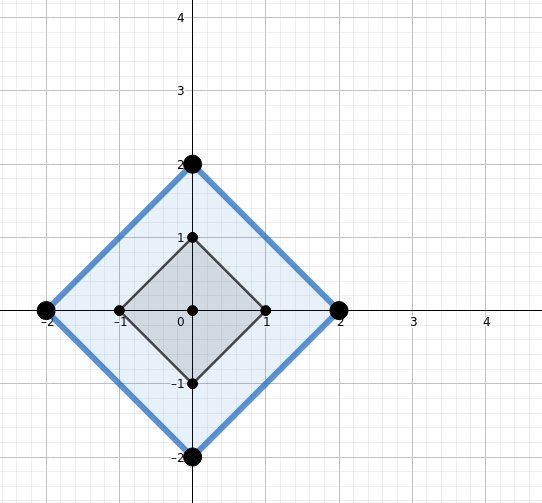}
    \caption{$\widetilde{P}_{\partial (\sigma_3)}$ and its interior polytope $Q_{\partial (\sigma_{3})}$ translated to the origin.}
    \label{fig: Interior2}
\end{figure}

Surprisingly, it turns out that $P_{\partial (\sigma_{d+1})}$ and its interior polytope are combinatorially equivalent. More precisely, the following stronger statement is true: 

\begin{theorem}\label{thm: description interior polytope}
Let $d\in\N$ be even. Then the following statements hold:
\begin{itemize}
    \item[\emph{(a)}] The complete and irredundant facet description of $Q_{\partial (\sigma_{d+1})}$ is given by:
\begin{itemize}
	\item[\emph{(i)}] $\oneb^{\intercal}\cdot x\leq d+1$,
	\item[\emph{(ii)}] $\oneb_{\mathrm{odd}}^{\intercal} \cdot x-x_i\leq \frac{d}{2}$ for even $i\in [d]$,
	\item[\emph{(iii)}] $\oneb_{\mathrm{even}}^{\intercal} \cdot x-x_j\leq \frac{d}{2}$  for odd $j\in [d] $,
	\item[\emph{(iv)}] $x_i+x_{j}\geq 1$ for $1\le i<j\leq d$ such that $i+j$ is odd.
\end{itemize}
\item[\emph{(b)}] $Q_{\partial (\sigma_{d+1})}-\oneb$ is reflexive. In particular, $\oneb$ is the unique interior lattice point of $Q_{\partial (\sigma_{d+1})}$.
\item[\emph{(c)}] $2\cdot \left(Q_{\partial (\sigma_{d+1})}-\oneb\right)=\widetilde{P}_{\partial (\sigma_{d+1})}-\oneb.$
\end{itemize}

\end{theorem}
\begin{proof} 
We let $Q=\widetilde{P}_{\partial (\sigma_{d+1})}-\oneb$. The vertices of $Q$ are given by $u^{(\ell)}\coloneqq b^{(\ell)}+\oneb$ for $1\leq \ell\leq d+2$. It is immediate that all coordinates of $u^{(\ell)}$ are divisible by $2$. Hence, $\frac{1}{2}Q$ is a lattice polytope. Using \Cref{thm: complete facet discription}, it follows that the facets of $\frac{1}{2}Q$ are given by 
\begin{itemize}
	\item $\oneb^{\intercal}\cdot x\leq 1$,
	\item $\oneb_{\mathrm{odd}}^{\intercal} \cdot x-x_i\leq 1$ for even $i\in [d]$,
	\item $\oneb_{\mathrm{even}}^{\intercal} \cdot x-x_j\leq 1$  for odd $j\in [d]$,
	\item $x_i+x_{j}\geq - 1$ for $1\le i<j\leq d$ such that $i+j$ is odd,
\end{itemize}
which shows that $\frac{1}{2}Q$ is reflexive. It remains to show that $\frac{1}{2}Q+\oneb=Q_{\partial (\sigma_{d+1})}$. Since $\frac{1}{2}Q+\oneb$ is a lattice polytope, it follows that $\frac{1}{2}Q+\oneb\subseteq Q_{\partial (\sigma_{d+1})}$. For the other inclusion it suffices to note that the facets of $\frac{1}{2}Q$ and $\widetilde{P}_{\partial (\sigma_{d+1})}$ are parallel and that they have distance $\frac{1}{\sqrt{d}}$,  $\frac{\sqrt{2}}{\sqrt{d+2}}$,  $\frac{\sqrt{2}}{\sqrt{d+2}}$ and  $\frac{1}{\sqrt{2}}$ to each other for facets of the form in (i), (ii), (iii) and (iv), respectively. This implies that there is no lattice point in $\widetilde{P}_{\partial (\sigma_{d+1})}\setminus ((\frac{1}{2}Q+\oneb)\cup \partial\widetilde{P}_{\partial (\sigma_{d+1})})$ 
 and hence $Q_{\partial (\sigma_{d+1})}\subseteq \frac{1}{2}Q+\oneb$.
\end{proof}
 We define vectors $c^{(1)},\ldots,c^{(d+2)}\in\mathbb{R}^d$ by $c^{(\ell)}_k=\frac{d+2}{2}$ if $k=\ell-2$ and $c^{(\ell)}_k=\linebreak\max(0,(-1)^{k+\ell-1})$, otherwise. Combining \Cref{thm: Facet inequalities and vertices} and \Cref{thm: description interior polytope} (c), we get the following description of the vertices of $Q_{\partial (\sigma_{d+1})}$ and its facets. 

\begin{cor}\label{cor: matrix P I and interior lattice point}
	The vertices of $Q_{\partial (\sigma_{d+1})}$ are the vectors $c^{(1)},\ldots,c^{(d+2)}$. Moreover, the vertices, that attain equality  in  \Cref{thm: description interior polytope} (i)--(iv), are given by the sets $\{c^{(\ell)}~:~3\leq \ell\leq d+2\}$, $\{c^{(\ell)}~:~\ell\in [d+2]\setminus \{1,i+2\}\}$, $\{c^{(\ell)}~:~\ell\in [d+2]\setminus \{2,j+2\}\}$ and $\{c^{(\ell)}~:~\ell\in [d+2]\setminus \{i+2,j+2\}\}$, respectively.
\end{cor}

We now recall several definitions and facts concerning regular unimodular triangulations (see \cite[Subsection 2.3.2.]{UnimodularTriangulations} for more on these topics).

Given full-dimensional polytopes $P\subseteq \R^d$ and $P'\subseteq \R^{d'}$ of positive dimension, their \emph{join} $P\ast P'$ is the $(d+d'+1)$-dimensional polytope defined by
	\begin{equation*}
	P	\times \{\zerob_{d'}\}\times \{0\} \ \ \cup \ \ \{\zerob_{d}\}\times P'\times \{1\}.
	\end{equation*}

The next statement, which is well-known, will be crucial for the construction of a regular unimodular triangulation of $\widetilde{P}_{\partial (\sigma_{d+1})}$.
\begin{theorem}\label{thm: reg unimod triangulation of join}
Let $P\subseteq \R^d$ and $P'\subseteq \R^{d'}$ be polytopes of dimension $d$ and $d'$, respectively. Let $S=\{S_i\ :\ i\in[n]\}$ and $S'=\{S'_j\ :\ j\in[m]\}$ be triangulations of $P$ and $P'$, respectively, where $S_i$ and $S'_j$ denote the full-dimensional cells. If both $S$ and $S'$ are regular and unimodular, then 
	\begin{equation*}
		\Tc=\left\{S_i\ast S'_j~:~i\in[n], j\in [m]\right\}
	\end{equation*}
	 is a regular unimodular triangulation of $P\ast P'$. 
\end{theorem}

We will also make use of the following statement, see \cite[Theorem 4.8]{UnimodularTriangulations}.
\begin{theorem}\label{thm: unimod triang dilation}
	If $P$ has a (regular) unimodular triangulation $\Tc$, then so has any dilation $cP$, where $c$ is a positive integer.
\end{theorem}

 A well-studied subdivision, which is related to the Veronese construction in algebra but also appears in topology \cite{BRENTI2009545,Edelsbrunner,Grayson,Roemer}, is the so-called \emph{$r$\textsuperscript{th} edgewise subdivision} of a simplicial complex. In the following, we review this definition for the special case that $\Delta$ is the $(n-1)$-dimensional simplex on vertex set $V=\{\mathbf{e}_1,\mathbf{e}_2,\ldots,\mathbf{e}_n\}\subseteq \R^n$. For a positive integer $r$, let $\Omega_r=\{(i_1,\ldots,i_n)\in\N^n~:~i_1+i_2+\cdots+i_n=r\}$ denote the set of lattice points $r\Delta\cap \mathbb{Z}^n$. For $x=(x_1,\ldots,x_n)\in\Z^n$, we define
 \[
 \phi(x)\coloneqq (x_1,x_1+x_2,\ldots,x_1+\cdots +x_n)\in \R^n.
 \]
 The \emph{$r$\textsuperscript{th} edgewise subdivision} of $\Delta$ is the simplicial complex $\mathrm{esd}_r(\Delta)$ on vertex set $\Omega_r$, for which $F\subseteq\Omega_r$ is a face if for all $x,y\in F$
 \[
 \phi(x)-\phi(y)\in\{0,1\}^n \qquad \mbox{or} \qquad \phi(y)-\phi(x)\in\{0,1\}^n.
 \]
 By definition, the geometric realization of the $r$\textsuperscript{th} edgewise subdivision of $\Delta$ gives a lattice triangu\-la\-tion of $r\Delta$. It is known that this triangulation is regular \cite[Proposition 6.4.]{Roemer}, which is also unimodular since all maximal simplices have normalized volume $1$. In the following, we will use $\mathrm{esd}_r(\Delta)$ to denote both, the triangulation as a simplicial complex and its geometric realization.
 Given any $(n-1)$-dimensional unimodular simplex $\Gamma\subseteq \mathbb{R}^n$, $\mathrm{esd}_r(\Delta)$, naturally induces a regular unimodular triangulation of $r\Gamma$ (by applying the corresponding unimodular transformation). Slightly abusing notation, we will refer to this triangulation as edgewise subdivision of $\Gamma$ or even of $r\Gamma$, denoted $\mathrm{esd}_r(\Gamma)$. Moreover, the restriction of $\mathrm{esd}_r(\Gamma)$ to any face $F\in\Gamma$ equals $\mathrm{esd}_r(F)$ as a simplicial complex and as geometric realization.
 \begin{exa}
\Cref{fig: Bsp_esd} depicts the $3$\textsuperscript{rd} edgewise subdivision of the $2$-dimensional simplex $\Delta_2\coloneqq\conv((0,0),(1,0),(0,1))$ as triangulation of $3\Delta_2$. The vertex labels correspond to the vertex labels from the original definition and not the lattice points. 
\begin{figure}[h]
    \centering
    \includegraphics[width=6.0cm]{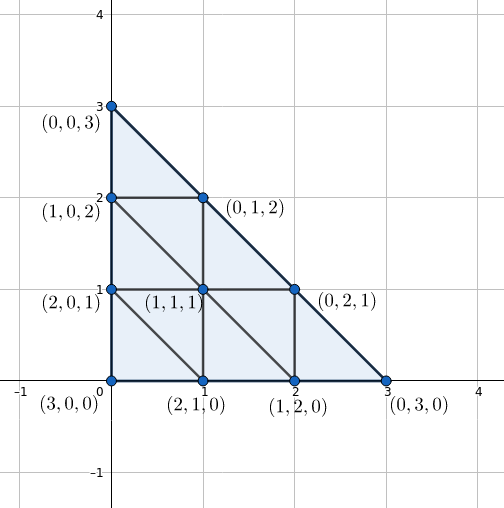}
    \caption{The triangulation of $3\cdot\Delta_2$ given by  $\mathrm{esd}_3(\Delta_2)$.}
    \label{fig: Bsp_esd}
\end{figure}
\end{exa}
We now outline our strategy to show that $\Ptd$ has a regular unimodular triangulation. 
We first prove that $\Ptd$ and the facets of $Q_{\partial (\sigma_{d+1})}$, if $d$ is odd and even, respectively, are unimodular equivalent to joins of dilated standard simplices. If $d$ is odd and even, we can hence  triangulate $\Ptd$ and facets of $Q_{\partial (\sigma_{d+1})}$ as join of edgewise subdivisions. If $d$ is odd, the claim follows by \Cref{thm: reg unimod triangulation of join}. If $d$ is even, we next show that these triangulations are consistent on intersections of facets. By coning with $\oneb$, we get a unimodular triangulation of $Q_{\partial (\sigma_{d+1})}$ (see \Cref{thm: description interior polytope} (d)) and hence of $\Ptd$ by \Cref{thm: unimod triang dilation}. The regularity follows by using that the triangulation is regular on single facets and that each facet is triangulated in the same way. 

The next statement yields the first step in the outlined strategy. 
\begin{prop}\label{prop: facets as join}
    \begin{itemize}
        \item[(a)] Let $d\geq 2$ be even and $F\in\Fc\left(Q_{\partial (\sigma_{d+1})}\right)$. Then 
        $$
        F\cong \left(\frac{d+2}{2}\Delta_{\frac{d-2}{2}}-\oneb_{\frac{d-2}{2}}\right)\ast \left(\frac{d+2}{2}\Delta_{\frac{d-2}{2}}-\oneb_{\frac{d-2}{2}}\right).
        $$
        \item[(b)] Let $d\geq 1$ be an odd integer. Then
        \begin{equation*}
	       \Ptd\cong \big((d+2)\Delta_{\frac{d+1}{2}}-2\cdot\oneb\big)\ast \big((d+2)\Delta_{\frac{d-1}{2}}-2\cdot\oneb\big).
        \end{equation*}
    \end{itemize} 
\end{prop}
\begin{proof}
The proof of (a) is divided into four cases, according to the four classes of facets from \Cref{thm: description interior polytope} (a).

Let $F=\{x\in \R^{d}~:~\oneb^{\intercal}\cdot x\leq d+1\}$. By \Cref{cor: matrix P I and interior lattice point}, the vertices of $F$ are $c^{(3)},\ldots,c^{(d+2)}$. We now consider the matrix $A$, whose $\ell$-th column equals $c^{(2\ell+1)}$ if $1\leq \ell\leq \frac{d}{2}$ and $c^{(2\ell+2-d)}$ if $\frac{d}{2}+1\leq \ell\leq d$. If we reorder the rows of $A$, by taking first the rows with odd index and then the ones with even index, increasingly, we obtain a matrix $S$, which looks as follows: \[
S=\left(\begin{array}{c|c}
	\frac{d+2}{2}\cdot E_{\frac{d}{2}}& \oneb_{\frac{d}{2}\times\frac{d}{2}}\\ \hline
	 \oneb_{\frac{d}{2}\times\frac{d}{2}}&\frac{d+2}{2}\cdot E_{\frac{d}{2}}
\end{array}\right),
\]
where $ \oneb_{k\times k}$ denotes the $(k\times k)$-matrix with all entries equal to $1$. 
Clearly, $F\cong \conv(S)$. Let  $E'_k\in\Z^{(k-1)\times k}$ be the $(k\times k)$-identity matrix with its first row removed and let
\begin{equation*}
U=\left(\begin{array}{ccc|ccc}
	&E'_{\frac{d}{2}} & & &\zerob_{\frac{d-2}{2}\times\frac{d}{2}} &  \\ \hline
	& \zerob_{\frac{d-2}{2}\times\frac{d}{2}} & & &E'_{\frac{d}{2}} &  \\ \hline
	0&\cdots & 0& 1& \cdots&  1\\
	1& \cdots &1 &1  &\cdots  &  1\\
\end{array}\right)\in\Z^{d\times d}.
\end{equation*}
It is easily seen that $U$ is unimodular and a direct computation shows that
\begin{equation*}
U\cdot (S-\oneb_{d\times d})=\left(\begin{array}{ccc|ccc}
	& \Mcc\left(\frac{d+2}{2}\Delta_{\frac{d-2}{2}}-\oneb_{\frac{d-2}{2}\times\frac{d}{2}} \right)& & &\zerob_{\frac{d-2}{2}\times\frac{d}{2}}  &	\\ \hline
	&\zerob_{\frac{d-2}{2}\times\frac{d}{2}}  & & & \Mcc\left(\frac{d+2}{2}\Delta_{\frac{d-2}{2}}-\oneb_{\frac{d-2}{2}\times\frac{d}{2}} \right)&	\\ \hline
	0& \cdots&0 &1 & \cdots&1	\\ 
		1& \cdots &1 &1  &\cdots  &  1\\
\end{array}\right),
\end{equation*}
where $\zerob_{k\times k}$ denotes the $(k\times k)$-matrix with all entries equal to $0$ and  $\Mcc\left(\frac{d+2}{2}\Delta_{\frac{d-2}{2}}-\oneb_{\frac{d-2}{2}\times\frac{d}{2}} \right) $ 
 denotes the matrix whose columns are the vertices of $\frac{d+2}{2}\Delta_{\frac{d-2}{2}}-\oneb_{\frac{d-2}{2}\times\frac{d}{2}}$ in the obvious order. Since $F\cong \conv( U \cdot (S-\oneb_{d\times d}))$, the claim follows after projection on the first $d-1$ coordinates and by the definition of the join. We also note that the vertices of $F$  corresponding to the vertices of the dilated simplices are $\{c^{(2\ell+1)}~\colon~1\leq \ell \leq \frac{d}{2}\}$ and $\{c^{(2\ell)}~\colon~2\leq \ell \leq \frac{d}{2}+1\}$.

 Similarly, one can show that for the facets defined by
 \begin{itemize}
 \item $\oneb_{\mathrm{odd}}^{\intercal} \cdot x-x_i\leq \frac{d}{2}$, where $i\in [d]$ is even, 
\item $\oneb_{\mathrm{even}}^{\intercal} \cdot x-x_j\leq \frac{d}{2}$, where $j\in [d] $ is odd,
\item $x_i+x_{j}\geq 1$ for $1\le i<j\leq d$ such that $i+j$ is odd,
\end{itemize}
respectively, the vertices 
\begin{itemize}
 \item $\{c^{(2\ell+1)}~\colon~1\leq \ell \leq \frac{d}{2}\}$  and $\{c^{(2\ell)}~\colon~1\leq \ell \leq \frac{d}{2}+1, \ell\neq \frac{i+2}{2}\}$,
\item $\{c^{(2\ell+1)}~\colon~0\leq \ell \leq \frac{d}{2}, \ell\neq \frac{j+1}{2}\}$  and $\{c^{(2\ell)}~\colon~2\leq \ell \leq \frac{d}{2}+1, \ell\neq \frac{i+2}{2}\}$,
 \item $\{c^{(2\ell+1)}~\colon~0\leq \ell \leq \frac{d}{2}, \ell\neq \frac{j+1}{2}\}$  and $\{c^{(2\ell)}~\colon~1\leq \ell \leq \frac{d}{2}+1, \ell\neq \frac{i+2}{2}\}$,
 \end{itemize}
 respectively, correspond to the vertices of the dilated simplices. The rather technical proofs can be found in the appendix.
 Similarly, (b) will be shown in the appendix. 
\end{proof}

We recall and prove \Cref{thm: r.u.t of P}.
\TheoremA*

\begin{proof}
Since $\Ptd$ and $P_{\partial(\sigma_{d+1})}$ are unimodular equivalent, it suffices to show the statement for $\Ptd$. First assume that $d$ is odd. 
   By \Cref{prop: facets as join} (b), we know that 
    \[
    \Ptd\cong \big((d+2)\Delta_{\frac{d+1}{2}}-2\cdot\oneb\big)\ast \big((d+2)\Delta_{\frac{d-1}{2}}-2\cdot\oneb\big).
    \]
 Since the $(d+2)$\textsuperscript{nd} edgewise subdivision is a  regular unimodular triangulation of the $(d+2)$\textsuperscript{nd} dilation of any unimodular simplex (as well as of any translation), we conclude with \Cref{thm: reg unimod triangulation of join} that $\Ptd$ has a regular unimodular triangulation.

Next assume that $d$ is even.  If $d=0$, $\Ptd$ is just a point and there is nothing to show. 

    Let $d\geq 2$. We construct a regular unimodular triangulation of the interior polytope $Q_{\partial (\sigma_{d+1})}$. By \Cref{prop: facets as join} every facet $F$ of  $Q_{\partial (\sigma_{d+1})}$ is unimodular equivalent to 
    \begin{equation}\label{eqn: join facets}
        \left(\frac{d+2}{2}\Delta_{\frac{d-2}{2}}-\oneb_{\frac{d-2}{2}}\right)\ast \left(\frac{d+2}{2}\Delta_{\frac{d-2}{2}}-\oneb_{\frac{d-2}{2}}\right).
    \end{equation}
    By the same reasoning as for $d$ odd, we can triangulate $F$ as join of edgewise subdivisions of unimodular simplices. In this way, we obtain regular unimodular triangulations of each facet of  $Q_{\partial (\sigma_{d+1})}$. We now show that the union of these triangulations, yields a triangulation of the boundary of $Q_{\partial (\sigma_{d+1})}$. For this aim, let $F$ and $G$ be facets of  $Q_{\partial (\sigma_{d+1})}$ and let $\mathcal{T}(F)$ and $\mathcal{T}(G)$ be the considered triangulations. Let us further denote by $F_i$ and $G_i$, where $i\in [2]$, the vertex sets corresponding to the vertex sets of the dilated (and translated) simplices in \eqref{eqn: join facets}. It follows from the end of the proof of \Cref{prop: facets as join} that (after possible renumbering)
    \[
    (F_1\cup F_2)\cap (G_1\cup G_2)=(F_1\cap F_2)\cup (G_1\cap G_2).
    \]
This directly yields that the restrictions of $\mathcal{T}(F)$ and $\mathcal{T}(G)$ to $F\cap G$ coincide: Indeed, they are given as the join of the edgewise subdivisions of the dilated (and translated) simplices on vertex sets $F_1\cap F_2$ and $G_1\cap G_2$. This shows that the union of the triangulations of the facets is indeed a triangulation of the boundary of $Q_{\partial (\sigma_{d+1})}$, which is, in particular, unimodular. Since, by \Cref{thm: description interior polytope} (b), $Q_{\partial (\sigma_{d+1})}-\oneb$ is reflexive, we can extend this triangulation to a unimodular triangulation of $Q_{\partial (\sigma_{d+1})}$ by coning over the unique interior lattice point $\oneb$. In the following, we call this triangulation $\mathcal{T}$.

It remains to show that $\mathcal{T}$ is a regular triangulation. The previous paragraph implies that the induced triangulations on facets $Q_{\partial (\sigma_{d+1})}$ are all regular and unimodular equivalent to each other. In particular, there exists a simultaneous lifting function $h$ yielding the triangulation of an arbitrary facet. Fix a facet $F$ and let $\mathcal{T}(F)$ be the induced triangulation on $F$. Since $F$ is a simplex, we can assume that $h(v)=1$ for any vertex $v\in F$. Moreover, for any lattice point $u$ in $F$, that is not a vertex, we have $h(u)<1$, since otherwise $u$ would not be  a vertex of $\mathcal{T}(F)$. Hence, there exists a non-negative function $g$, whose values are bounded by $1$, that vanishes on the vertices of $F$ such that $h=1-g$. Moreover, for any $\epsilon>0$, $h_\epsilon=1-\epsilon g$ is also a lifting function for $F$ yielding $\mathcal{T}(F)$. Finally, ignoring $\oneb$ and lifting all other lattice points in $Q_{\partial (\sigma_{d+1})}$ according to the simultaneous lifting function $h_\epsilon$, gives a lifting function such that the projection of the lower envelope yields $\mathcal{T}$ on the boundary of $Q_{\partial (\sigma_{d+1})}$ and potentially additional faces in the interior. Lifting $\oneb$ at height $0$, gives a lifting of all lattice points of $Q_{\partial (\sigma_{d+1})}$. If $\epsilon$ is sufficiently small, one can guarantee that the triangulation obtained as the lower envelope is the cone with $\oneb$ over the boundary of the previous triangulation (ignoring $\oneb$) since potential interior faces that we had seen before, do no longer  lie in the lower envelope. 

The claim follows by \Cref{thm: description interior polytope} (c) and \Cref{thm: unimod triang dilation}.
\end{proof}
Analyzing the proof of \Cref{thm: r.u.t of P}, we can compute the normalized volume of $P_{\partial(\sigma_{d+1})}$:
\begin{cor}\label{cor: normalized volume}
The normalized volume of $P_{\partial(\sigma_{d+1})}$ is $(d+2)^d$. 
\end{cor}
\begin{proof}
We compute the normalized volume of $\Ptd$, which equals the one of  $P_{\partial(\sigma_{d+1})}$, by counting the number of maximal simplices in the unimodular triangulation $\mathcal{T}$ constructed in the proof of \Cref{thm: r.u.t of P}. 

First assume that $d$ is odd. We have seen that $\mathcal{T}$ is unimodular equivalent to 
\[
\mathrm{esd}_{d+2}\left(\Delta_{\frac{d+1}{2}}\right)\ast \mathrm{esd}_{d+2}\left(\Delta_{\frac{d-1}{2}}\right).
\]
Since the $r$\textsuperscript{th} edgewise subdivision of an $m$-simplex, has $r^m$ maximal simplices, it follows that the number of maximal simplices in the constructed unimodular triangulation of $\Ptd$ equals
\[
(d+2)^{\frac{d-1}{2}}\cdot (d+2)^{\frac{d+1}{2}}=(d+2)^d.
\]

Let $d$ be even. We first compute the normalized volume of $Q_{\partial (\sigma_{d+1})}$. Combining \Cref{thm: complete facet discription} and 
\Cref{thm: description interior polytope}, it follows that $Q_{\partial (\sigma_{d+1})}$ has exactly $\frac{(d+2)^2}{4}$ facets. By the proof of  \Cref{thm: r.u.t of P} each of these has a unimodular triangulation that is unimodular equivalent to 
\[
\mathrm{esd}_{\frac{d+2}{2}}\left(\Delta_{\frac{d-2}{2}}\right)\ast \mathrm{esd}_{\frac{d+2}{2}}\left(\Delta_{\frac{d-2}{2}}\right).
\]
As in the case that $d$ is odd, we conclude that each facet is triangulated into $\left(\frac{d+2}{2}\right)^{\frac{d-2}{2}}\cdot \left(\frac{d+2}{2}\right)^{\frac{d-2}{2}} \linebreak=\left(\frac{d+2}{2}\right)^{d-2}$ many maximal simplices and hence $Q_{\partial (\sigma_{d+1})}$ has normalized volume
$\frac{(d+2)^d}{2^d}$. Since, by \Cref{thm: description interior polytope} (c),  $\Ptd+\oneb= 2\cdot Q_{\partial (\sigma_{d+1})}$, the claim follows.
\end{proof}


\subsection{Unimodality and real-rootedness}
The goal of this subsection is to prove \Cref{thm: unimodality real-rooted}.

If $d$ is even, then by the proof of \Cref{thm: r.u.t of P}, $Q_{\partial (\sigma_{d+1})}$ has a regular unimodular triangulation. Since it is also  reflexive (after translation) by \Cref{thm: description interior polytope} (b), the next statement is immediate from \cite[Theorem 1]{BrunsRoemer2007} (see also \cite[Theorem 1.3]{Athanasiadis}):

\begin{lem}\label{lem: unimodality Interior polytope}
    Let $d$ be an even positive integer. Then $h^\ast(Q_{\partial (\sigma_{d+1})})$ is symmetric and unimodal.
\end{lem}

To show unimodality of $h^\ast (\Ptd)$, if $d$ is even, we need to analyze the change of the $h^\ast$-vector under the second dilation of a polytope (cf., \Cref{thm: description interior polytope} (c)). 
Given a $d$-dimensional lattice polytope $P$, it follows, e.g., from \cite[Theorem 1.1]{BRENTI2009545}  (see also \cite{Beck,VeroneseConstruction}) that
\begin{equation}\label{eqn: h i star 2nd dilation}
    h_i^\ast(2P)=\sum_{j=0}^{d} \binom{d+1}{2i-j}h_j^\ast(P).
\end{equation}

We need the following technical but crucial lemma.

\begin{lem}\label{lem: symmetrie r_j}
Let $i\in\N$ and $r_j\coloneqq\binom{d+1}{2i+2-j}-\binom{d+1}{2i-j}$. Then  for $k\in\N$, we have
\[
-r_{\lceil2i+2-\frac{d+3}{2} \rceil-k}=r_{\lfloor2i+2-\frac{d+3}{2} \rfloor+k}.
\]
\end{lem}
\begin{proof}
    We set
    $a_j=\binom{d+1}{2i+2-j}$ and $b_j=\binom{d+1}{2i-j}$. 
    The claim follows if both
    \[
    a_{\lceil2i+2-\frac{d+3}{2} \rceil-k}=b_{\lfloor2i+2-\frac{d+3}{2} \rfloor+k} \quad \mbox{and} \quad b_{\lceil2i+2-\frac{d+3}{2} \rceil-k}=a_{\lfloor2i+2-\frac{d+3}{2} \rfloor+k}
    \]
    hold. Due to the symmetry of the binomial coefficient it suffices to show that
    \begin{itemize}
    \item[(i)]
        $(2i+2-\lceil2i+2-\frac{d+3}{2} \rceil+k)+(2i-\lfloor2i+2-\frac{d+3}{2} \rfloor-k)=d+1$ 
    \item[(ii)]
        $(2i-\lceil2i+2-\frac{d+3}{2} \rceil+k)+(2i+2-\lfloor2i+2-\frac{d+3}{2} \rfloor-k)=d+1.$
    \end{itemize}
    It is obvious that (i) and (ii) are equivalent. The claim follows from direct computations.
\end{proof}

The next statement will be the key ingredient to show that $h^\ast(\Ptd)$ is unimodal.

\begin{prop}\label{lem: first half h 2nd dilation increasing}
Let  $b=(b_0,\ldots,b_d)$ be a symmetric and unimodal sequence of non-negative reals. Let $c=(c_0,\ldots,c_d)$ be defined by 
\[
    c_i=\sum_{j=0}^{d} \binom{d+1}{2i-j}b_j.
\] 
Then,
\[
c_{0}\leq c_1\leq \cdots\leq c_{\lfloor\frac{d+1}{2}\rfloor}.
\]
\end{prop}
\begin{proof}
    We define $r_j$ as in \Cref{lem: symmetrie r_j}. Note that $r_j\geq 0$ if and only if $j\geq 2i+2-\frac{d+3}{2}$. 
    For $0\le i< \frac{d+1}{2}$, we have
    \begin{align*}
    c_{i+1}-c_i=&\sum_{j=0}^{d} \left[\binom{d+1}{2i+2-j}-\binom{d+1}{2i-j}\right]b_j
    = \sum_{j=0}^{d} r_jb_j\\
    =& \sum_{j=0}^{2(2i+2-\frac{d+3}{2})} r_jb_j+\sum_{j=2(2i+2-\frac{d+3}{2})+1}^d r_jb_j\\
    =&\sum_{j=1}^{\lceil 2i+2-\frac{d+3}{2}\rceil}r_{\lfloor 2i+2-\frac{d+3}{2}\rfloor+j}\left(b_{\lfloor 2i+2-\frac{d+3}{2}\rfloor+j}-b_{\lceil 2i+2-\frac{d+3}{2}\rceil-j}\right)\\
    &+r_{2i+2-\frac{d+3}{2}}b_{2i+2-\frac{d+3}{2}}+\sum_{j=2(2i+2-\frac{d+3}{2})+1}^d r_jb_j,
    \end{align*}
     where for the last equality, we use \Cref{lem: symmetrie r_j} and we set  $+r_{2i+2-\frac{d+3}{2}}b_{2i+2-\frac{d+3}{2}}=0$ if $d$ is even. 
    Since $b_j\geq 0$ and $r_j\geq 0$ for$j\geq 2i+2-\frac{d+3}{2}$, it follows that the single summand and the sum in the last line of the above computation are both non-negative. Concerning the first sum, the coefficients $r_{2i+2-\frac{d+3}{2}+j}$ are non-negative and therefore, in order to show non-negativity of $c_{i+1}-c_i$, it suffices to show that for $1\leq j\leq \lceil 2i+2-\frac{d+3}{2}\rceil$, we have
    \[
   b_{\lfloor 2i+2-\frac{d+3}{2}\rfloor+j}\geq b_{\lceil 2i+2-\frac{d+3}{2}\rceil-j}.
    \]
    This directly follows from the unimodality and symmetry of the sequence $b$ if $2i+2-\frac{d+3}{2}+j\leq \frac{d+1}{2}$. Assume $2i+2-\frac{d+3}{2}+j> \frac{d+1}{2}$. Since $i\leq \frac{d}{2}$, we have 
    \[
    \frac{d+1}{2}< 2i+2-\frac{d+3}{2}+j\leq d+2-\frac{d+3}{2}+j=\frac{d+1}{2}+j\leq\left \lfloor \frac{d+1}{2}\right\rfloor+j.
    \]
Using that $b$ is symmetric and unimodal it follows that 
    \[
    b_{\lfloor 2i+2-\frac{d+3}{2}\rfloor+j}\geq b_{\left \lfloor \frac{d+1}{2}\right\rfloor+j}=b_{d-\left\lfloor\frac{d}{2}\right\rfloor-j}\geq b_{\lceil 2i+2-\frac{d+3}{2}\rceil-j}.
    \]
    This shows the claim.
\end{proof}
We now recall and prove \Cref{thm: unimodality real-rooted}:

\TheoremB*

\begin{proof}
Since $P_{\partial(\sigma_{d+1})}$ has a regular unimodular triangulation $\mathcal{T}$ by \Cref{thm: r.u.t of P}, we have\\ $h^\ast(P_{\partial(\sigma_{d+1})})=h(\mathcal{T})$. If $d$ is odd, such a triangulation is given by 
\[
\mathrm{esd}_{d+2}\left(\Delta_{\frac{d+1}{2}}\right)\ast \mathrm{esd}_{d+2}\left(\Delta_{\frac{d-1}{2}}\right)
\]
and its $h$-polynomial equals $h\left(\mathrm{esd}_{d+2}\left(\Delta_{\frac{d+1}{2}}\right);t\right)\cdot h\left(\mathrm{esd}_{d+2}\left(\Delta_{\frac{d-1}{2}}\right);t\right)$. Since both factors are real-rooted by \cite[Corollary 4.4]{VeroneseConstruction}, so is $h^\ast\left(P_{\partial(\sigma_{d+1})};t\right)$.

Suppose that $d$ is even. Combining \Cref{lem: unimodality Interior polytope}, \eqref{eqn: h i star 2nd dilation} and \Cref{lem: first half h 2nd dilation increasing}, we get that $h^\ast(P_{\partial(\sigma_{d+1})})$ is increasing up to the middle, i.e., 
    \[
    h_0^\ast(P_{\partial(\sigma_{d+1})})\leq h_1^\ast(P_{\partial(\sigma_{d+1})})\leq\cdots\le h_{\frac{d}{2}}^\ast(P_{\partial(\sigma_{d+1})}).
    \]
    Since, by \Cref{thm: r.u.t of P}, $P_{\partial(\sigma_{d+1})}$ has a regular unimodular triangulation it follows by \cite[Theorem 1.3]{Athanasiadis} that $h^\ast(P_{\partial(\sigma_{d+1})})$ is decreasing beyond the middle, i.e.,
    \[
    h_{\frac{d}{2}}^\ast(P_{\partial(\sigma_{d+1})})\geq \cdots\geq h_d^\ast(P_{\partial(\sigma_{d+1})}).
    \]
The claim follows.
\end{proof}
We would like to remark that even though the interior polytope $Q_{\partial (\sigma_{d+1})}$ has a symmteric $h^\ast$-vector, this is not true for $P_{\partial(\sigma_{d+1})}$.


\section{Open problems}\label{sec: Questions} 
We end this article with some obvious directions for future research. 

We have initiated the study of Laplacian polytopes $P^{(i)}_\Delta$ by studying the special case that $\Delta$ is the boundary of a $(d+1$)-simplex and $i=d$. It is therefore natural to consider the following very general problem. 
\begin{problem}\label{prob: general}
Study geometric and combinatorial properties of $P^{(i)}_\Delta$ for (classes of) simplicial complexes and general $0\leq i\leq \dim\Delta$. In particular: What is the normalized volume? When do these polytopes have a regular unimodular triangulation? What properties do the $h^\ast$-vector and the $h^\ast$-polynomial have? 
\end{problem}
In view of \Cref{prop: L full rank P fd-1-simplex}, a good starting point might be to study $P^{(d)}_\Delta$ for simplicial $d$-balls, since in this case we already know that $P^{(d)}_\Delta$ is a simplex. As part of this problem, it might be useful to consider how Laplacian polytopes change under certain operations on the simplicial complex, e.g., deletion/contraction of vertices, taking links, connected sums, joins. We want to remark that for $i=1$ we get Laplacian simplices as studied in \cite{LaplSimplBraun2017} and \cite{meyer2018laplacian}.

We have shown that $P_{\partial(\sigma_{d+1})}$ has a regular unimodular triangulation by explicitly constructing one. However, for more general classes of simplicial complexes, a better approach might be to compute a Gr\"obner basis of the toric ideal. This gives rise to the following problem whose solution would also contribute to \Cref{prob: general}:

\begin{problem}
Describe a Gr\"obner basis of the toric ideal of $P^{(i)}_\Delta$ in terms of the combinatorics of $\Delta$. When does there exist a squarefree Gr\"obner basis (giving rise to a regular unimodular triangulation)?
\end{problem}

We want to emphasize that the Laplacian polytope depends on the ordering of the vertices of $\Delta$ (see \Cref{ex: ordering}). It is therefore natural to ask the following question:
\begin{question}
Which orderings yield (up to unimodular or combinatorial equivalence) the same Laplacian polytope? How many equivalence classes are there?
\end{question}

Apart from these more general problems, there are several open questions that are directly related to our results. In \Cref{cor: normalized volume}, we have computed the normalized volume of $P_{\partial(\sigma_{d+1})}$ explicitly and thereby have obtained a precise formula for the sum of the $h^\ast$-vector entries. Using the explicit regular unimodular triangulation from \Cref{thm: r.u.t of P} and inclusion-exclusion we can also express the $h^\ast$-polynomial as alternating sum, where all summands are products of $h^\ast$-polynomials of edgewise subdivisions of dilated simplices of varying dimension. Note that for $d$ odd, we only have one summand. However, this does not yield a direct combinatorial interpretation of the entries of the $h^\ast$-vector. We therefore propose the following problem:
\begin{problem}
Find a combinatorial interpretation of the entries of the $h^\ast$-vector of $P_{\partial(\sigma_{d+1})}$ (see \Cref{tab: hstar vectors} for the $h^\ast$-vectors if $1\leq d\leq 8$).
\end{problem}

\begin{table}[h]
    \centering
    \begin{tabular}{c|l}
      $d$ & $h^*\left(P_{\partial \sigma_{d+1}}\right)$   \\ \hline
       1  & $\left(1,2,0\right)$\\
       2  & $\left(1,10, 5\right)$\\
       3  & $\left(1,22,78,24,0\right)$\\
       4  & $\left(1,131, 726, 419,19\right)$\\
       5  & $\left(1,149,4049,8558,3750,300,0\right)$\\
       6  & $\left(1,  1478,38179,126372, 85623,10422,69\right)$\\
       7  & $(1,926,157566, 1135846, 2188310,1150800,145600,3920,0)$ \\
       8  & $(1, 17617, 1581403, 6864069, 43252570, 31729319, 6314903, 239867, 251)$\\
    \end{tabular}
    \caption{The $h^\ast$-vectors of $P_{\partial \sigma_{d+1}}$ for $d=1,\ldots,8$.}
    \label{tab: hstar vectors}
\end{table}

  Finally, in view of \Cref{thm: unimodality real-rooted} (a), we have the following conjecture:

\begin{con}
Let $d$ be even. Then, $h^\ast\left(P_{\partial(\sigma_{d+1})};x\right)$ is real-rooted.
\end{con}
We have verified this conjecture computationally up to $d=10$. For this problem, we suspect that an approach via interlacing sequences might be helpful, but we have not been able to carry it out so far.


\section{Appendix}
We provide the missing parts of the proof of \Cref{prop: facets as join}.  We recall some notation. We denote by $E'_k\in\Z^{(k-1)\times k}$ the $(k\times k)$-identity matrix with its first row removed and by $\oneb_{m\times n}$ and $\zerob_{m\times n}$ the $(m\times n)$-matrices whose entries are all equal to $1$ and $0$, respectively.  Moreover, we denote by $\Mcc\left(\frac{d+2}{2}\Delta_{\frac{d-2}{2}}-\oneb_{\frac{d-2}{2}\times\frac{d}{2}} \right)$ the matrix whose columns are the vertices of $\frac{d+2}{2}\Delta_{\frac{d-2}{2}}-\oneb_{\frac{d-2}{2}\times\frac{d}{2}}$ in the obvious order.

\begin{proof}[Proof of \Cref{prop: facets as join} (a)]
Let $d\geq 2$ and for a fixed even integer $i\in [d]$, consider the facet $F=\{x\in \R^d~:~\oneb_{\mathrm{odd}}^{\intercal} \cdot x-x_i\leq\frac{d}{2}\}$ of $Q_{\partial (\sigma_{d+1})}$. By \Cref{cor: matrix P I and interior lattice point}, the vertices of $F$ are $\{c^{(\ell)}~\colon~\ell\in[d+2]\setminus\{1,i+2\}\}$. We now consider the matrix $B\in\Z^{d\times d}$ whose $\ell$-th column equals $c^{(2\ell+1)}$ if $1\leq \ell\leq \frac{d}{2}$ and $c^{(2\ell-d)}$ if $\frac{d}{2}+1\leq \ell\leq \frac{i+d}{2}$ and $c^{(2\ell+2-d)}$ if $\frac{i+d}{2}+1\leq \ell\leq d$. 
If we reorder the rows of $B$, by taking first the rows with odd index, increasingly, followed by the row with index $i$ and then the remaining rows with even index, increasingly,  we obtain a matrix $S$, which looks as follows:
\[
S=\left(\begin{array}{ccc|ccc}
	& E_{\frac{d}{2}}\cdot\frac{d+2}{2} & & & \oneb_{\frac{d}{2}\times\frac{d}{2}}&\\ \hline
	1&\cdots&1&0&\cdots&0\\ \hline
		&\oneb_{\frac{d-2}{2}\times\frac{d}{2}} & & &  E'_{\frac{d}{2}}\cdot\frac{d+2}{2}&\\ 	
\end{array}\right).
\]
Clearly, $F\cong \conv(T)$. Let
\begin{equation*}
U=\left(\begin{array}{ccr|ccc}
			&E'_{\frac{d}{2}} & & & \zerob_{\frac{d-2}{2}\times\frac{d}{2}}&    \\ \hline
			& \zerob_{\frac{d-2}{2}\times\frac{d}{2}}& & &E'_{\frac{d}{2}} &   \\ \hline
		0 &\cdots & 0& -1\ 0& \cdots& 0 \\ \hline 
		 1&\cdots &1 &-1 \ 0&\cdots  &  0\\
		\end{array}\right)\in\Z^{d\times d}.
\end{equation*}
It is easy to see that $U$ is unimodular and a direct computation shows that
\begin{equation*}
U\cdot (S-\oneb_{d\times d})=\left(\begin{array}{ccc|ccc}
	& \Mcc\left(\frac{d+2}{2}\Delta_{\frac{d-2}{2}}-\oneb_{\frac{d-2}{2}\times\frac{d}{2}} \right)& & &\zerob_{\frac{d-2}{2}\times\frac{d}{2}}  &	\\ \hline
	&\zerob_{\frac{d-2}{2}\times\frac{d}{2}}  & & & \Mcc\left(\frac{d+2}{2}\Delta_{\frac{d-2}{2}}-\oneb_{\frac{d-2}{2}\times\frac{d}{2}} \right)&	\\ \hline
	0& \cdots&0 &1 & \cdots&1	\\ 
		1& \cdots &1 &1  &\cdots  &  1\\
\end{array}\right).
\end{equation*}
Since $F\cong \conv( U \cdot (S-\oneb_{d\times d}))$, the claim follows after projection on the first $d-1$ coordinates and by the definition of the join. We also note that the vertices of $F$  corresponding to the vertices of the dilated simplices are $\{c^{(2\ell+1)}~\colon~1\leq \ell \leq \frac{d}{2}\}$ and $\{c^{(2\ell)}~\colon~2\leq \ell \leq \frac{d}{2}+1,\ \ell\neq\frac{i+2}{2}\}$.\\

 For a fixed odd integer $j\in [d]$, consider the facet $G=\{x\in \R^d~:~\oneb_{\mathrm{even}}^{\intercal} \cdot x-x_j\leq \frac{d}{2}\}$ of $Q_{\partial (\sigma_{d+1})}$. 
 By \Cref{cor: matrix P I and interior lattice point}, the vertices of $F$ are $\{c^{(\ell)}~\colon~\ell\in[d+2]\setminus\{2,j+2\}\}$. We now consider the matrix $C\in\Z^{d\times d}$ whose $\ell$-th column equals $c^{(2\ell-1)}$ if $1\leq \ell\leq \frac{j+1}{2}$ and $c^{(2\ell+1)}$ if  $\frac{j+3}{2}\leq \ell\leq \frac{d}{2}$ and $c^{(2\ell+2-d)}$ if $\frac{d}{2}+1\le \ell\le d$. 
If we reorder the rows of $C$ by taking first the rows with odd index $k\in[d]\setminus \{j\}$, increasingly, followed by row $j$ and then the rows with even index, increasingly, we obtain a matrix $S$, which looks as follows: \[S=\left(\begin{array}{ccc|ccc}
	& E'_{\frac{d}{2}}\cdot\frac{d+2}{2} & & & \oneb_{\frac{d-2}{2}\times\frac{d}{2}}&\\ \hline
	0&\cdots&0&1&\cdots&1\\ \hline
		&\oneb_{\frac{d}{2}\times\frac{d}{2}} & & &  E_{\frac{d}{2}}\cdot\frac{d+2}{2}&\\ 	
\end{array}\right).
\]
 Clearly, $G\cong \conv(S)$. Let
\begin{equation*}
U=\left(\begin{array}{ccrccc}
			E_{\frac{d}{2}-1}&\Bigg| & & \zerob_{\frac{d-2}{2}\times\frac{d+2}{2}}&    \\ \hline
			& \zerob_{\frac{d-2}{2}\times\frac{d}{2}}&\Bigg| & E'_{\frac{d}{2}}  & \\ \hline
			0&\cdots & 0\ \big|& 1& \cdots& 1 \\ \hline 
			0\cdots& 0 &-1\Big| &1 &\cdots  &  1\\
		\end{array}\right)\in\Z^{d\times d}.
\end{equation*}
It is easy to see that $U$ is unimodular and a direct computation shows that
\begin{equation*}
U\cdot (S-\oneb_{d\times d})=\left(\begin{array}{ccc|ccc}
	& \Mcc\left(\frac{d+2}{2}\Delta_{\frac{d-2}{2}}-\oneb_{\frac{d-2}{2}\times\frac{d}{2}} \right)& & &\zerob_{\frac{d-2}{2}\times\frac{d}{2}}  &	\\ \hline
	&\zerob_{\frac{d-2}{2}\times\frac{d}{2}}  & & & \Mcc\left(\frac{d+2}{2}\Delta_{\frac{d-2}{2}}-\oneb_{\frac{d-2}{2}\times\frac{d}{2}} \right)&	\\ \hline
	0& \cdots&0 &1 & \cdots&1	\\ 
		1& \cdots &1 &1  &\cdots  &  1\\
\end{array}\right).
\end{equation*}
Since $G\cong \conv( U \cdot (S-\oneb_{d\times d}))$, the claim follows after projection on the first $d-1$ coordinates and by the definition of the join. We also note that the vertices of $G$  corresponding to the vertices of the dilated simplices are $\{c^{(2\ell+1)}~\colon~0\leq \ell \leq \frac{d}{2}, \ell\neq \frac{j+1}{2}\}$  and $\{c^{(2\ell)}~\colon~2\leq \ell \leq \frac{d}{2}+1, \ell\neq \frac{i+2}{2}\}$.\\


For fixed integers $1\leq i<j\leq d$ of different parity consider the facet $H=\{x\in \R^d~:~x_i+x_j\geq 1\}$ of $Q_{\partial (\sigma_{d+1})}$. Without loss of generality assume that $i$ is odd $j$ is even. By \Cref{cor: matrix P I and interior lattice point}, the vertices of $F$ are $\{c^{(\ell)}~\colon~\ell\in[d+2]\setminus\{i+2,j+2\}\}$.  We now consider the matrix $D\in\Z^{d\times d}$ whose $\ell$-th column equals $c^{(2\ell-1)}$ if $1\leq \ell\leq \frac{i+1}{2}$, $c^{(2\ell+1)}$ if  $\frac{i+3}{2}\leq \ell\leq \frac{d}{2}$, $c^{(2\ell-d)}$ if  $\frac{d}{2}+1\leq \ell\leq \frac{j+d}{2}$ and $c^{(2\ell+2-d)}$ if $\frac{j+d}{2}+1\le \ell\le d$. If we reorder the rows of $D$ by taking first the rows with odd index $k\in[d]\setminus\{i\}$, increasingly, followed by row $i$, followed by the rows with even index $\ell\in[d]\setminus \{j\}$, increasingly, followed by row $j$ as the last row, we obtain a matrix $S$, which looks as follows: \[
S=\left(\begin{array}{ccc|ccc}
	& \frac{d+2}{2}\cdot E'_{\frac{d}{2}}& & & \oneb'_{\frac{d}{2}}&\\ \hline
	0&\cdots&0&1&\cdots&1\\ \hline
		& \oneb'_{\frac{d}{2}}& & &\frac{d+2}{2}\cdot E'_{\frac{d}{2}}&\\ \hline
	1&\cdots&1&0&\cdots&0\\ 
\end{array}\right).
\]
 Clearly, $H\cong \conv(S)$. Let
\begin{equation*}
U=\left(\begin{array}{cccccr}
			&E_{\frac{d}{2}-1}\Bigg| & & \zerob^{\left(\frac{d}{2}-1\right)\times \left(\frac{d}{2}+1\right)}&    \\ \hline
			& \zerob'_{\frac{d}{2}}& &\Bigg|E_{\frac{d}{2}-1}&\Bigg|\zerob^{(\frac{d}{2}-1)\times 1}  &  \\ \hline
			& & & -e^\intercal_d& &  \\ \hline 
			& & & -(e_{\frac{d}{2}}+e_d)^\intercal & &  \\
		\end{array}\right)\in\Z^{d\times d}.
\end{equation*}
It is easy to see that $U$ is unimodular and a direct computation shows that
\begin{equation*}
U\cdot (S-\oneb_{d\times d})=\left(\begin{array}{ccc|ccc}
	& \Mcc\left(\frac{d+2}{2}\Delta_{\frac{d-2}{2}}-\oneb_{\frac{d-2}{2}\times\frac{d}{2}} \right)& & &\zerob_{\frac{d-2}{2}\times\frac{d}{2}}  &	\\ \hline
	&\zerob_{\frac{d-2}{2}\times\frac{d}{2}}  & & & \Mcc\left(\frac{d+2}{2}\Delta_{\frac{d-2}{2}}-\oneb_{\frac{d-2}{2}\times\frac{d}{2}} \right)&	\\ \hline
	0& \cdots&0 &1 & \cdots&1	\\ 
		1& \cdots &1 &1  &\cdots  &  1\\
\end{array}\right).
\end{equation*}
Since $H\cong \conv( U \cdot (S-\oneb_{d\times d}))$, the claim follows after projection on the first $d-1$ coordinates and by the definition of the join. We also note that the vertices of $H$  corresponding to the vertices of the dilated simplices are $\{c^{(2\ell+1)}~\colon~0\leq \ell \leq \frac{d}{2}, \ell\neq \frac{i+1}{2}\}$  and $\{c^{(2\ell)}~\colon~1\leq \ell \leq \frac{d}{2}+1, \ell\neq \frac{j+2}{2}\}$.
\end{proof}

\begin{proof}[Proof of \Cref{prop: facets as join} \emph{(ii)}]
 Let $d\geq1$ be an odd integer. We define vectors $u^{(1)},\ldots,\linebreak u^{d+2}\in\mathbb{R}^{d+1}$ by $u^{(\ell)}_k=d+1$ if $k=\ell-1$ and $u^{(\ell)}_k=(-1)^{k+\ell-1})$, otherwise.  By \Cref{lem: unimodular equivalence}, $u^{(1)},\ldots,u^{(d+2)}$ are the vertices of $\Ptd$. We now consider the matrix $E\in\Z^{(d+1)\times (d+2)}$ whose $\ell$-th column equals $b^{(2\ell-1)}$ if $1\leq \ell\leq \frac{d+3}{2}$ and $u^{(2\ell-(d+3))}$ if $\frac{d+3}{2}+1\le \ell\le d+2$. 
If we reorder the rows of $E$, by taking first the rows with even index and then the ones with odd index, increasingly, we obtain a matrix  $Q=(q_{k,\ell})\in\Z^{(d+1)\times (d+2)}$ with
\begin{itemize}
\item $q_{k,k+1}=d+1$ for $k\in [d+1]$, 
\item $q_{k,\ell}=1$ if $k\leq\frac{d+1}{2}$ and $\ell>\frac{d+3}{2}$, or $k>\frac{d+1}{2}$ and $\ell\leq\frac{d+3}{2}$
\item $q_{k,\ell}=-1$, otherwise.
\end{itemize}
Clearly, $\Ptd\cong \conv(Q)$. Let
\begin{equation*}	
    U=\left(\begin{array}{ccc|lcc}
		&E_{\frac{d+1}{2}} & & &\zerob_{\frac{d+1}{2}\times \frac{d+1}{2}} & \\ \hline
		& & &0 & & \\
		&\zerob_{\frac{d-1}{2}\times \frac{d+1}{2}}  & & \vdots&E_{\frac{d-1}{2}} & \\
		& & & 0& & \\\hline
		0&\cdots &0 &  1& \cdots&1 \\	
		\end{array}\right)\in\Z^{(d+1)\times(d+1)}.
	\end{equation*}
It is easy to see that $U$ is unimodular and a direct computation shows that
\begin{eqnarray*}
	U\cdot (Q-\oneb_{(d+1)\times (d+2)})=\left(\begin{array}{ccc|ccc}
			& \Mcc\left((d+2)\Delta_{\frac{d+1}{2}}-2\cdot\oneb\right)& & &\zerob &	\\ \hline
			&\zerob & & & \Mcc\left( (d+2)\Delta_{\frac{d-1}{2}}-2\cdot\oneb\right)&	\\ \hline
			0& \cdots&0 &1 & \cdots&1	\\ 
		\end{array}\right).
\end{eqnarray*} 
Since $\Ptd\cong \conv(U\cdot (Q-\oneb_{(d+1)\times (d+2)}))$, the claim follows by definition of the join.
\end{proof}



\bibliographystyle{plain}
\bibliography{references}

\begin{thebibliography}{10}

\bibitem{Adiprasito2022}
K.~Adiprasito, Stavros~A. Papadakis, V.~Petrotou, and J.~Steinmeyer.
\newblock Beyond positivity in ehrhart theory.
\newblock {\em Preprint arXiv: https://arxiv.org/abs/2210.10734}, 2022.

\bibitem{Athanasiadis}
C.A. Athanasiadis.
\newblock $h^*$-vectors, eulerian polynomials and stable polytopes of graphs.
\newblock {\em Electron. J. Combin.}, 11(2):Research Paper 6, 13 pp.
  (electronic), 2004/06.

\bibitem{LapSimpDigraphs}
G.~Balletti, T.~Hibi, M.~Meyer, and A.~Tsuchiya.
\newblock Laplacian simplices associated to digraphs.
\newblock {\em Arkiv för matematik}, 56, 12 2018.

\bibitem{BarahonaMahjoub1986}
F.~Barahona and A.~Mahjoub.
\newblock On the cut polytope.
\newblock {\em Mathematical Programming}, 36:157--173, 06 1986.

\bibitem{Beck}
M.~Beck and A.~Stapledon.
\newblock On the log-concavity of {H}ilbert series of {V}eronese subrings and
  {E}hrhart series.
\newblock {\em Math. Z.}, 264(1):195--207, 2010.

\bibitem{LaplSimplBraun2017}
B.~Braun and M.~Meyer.
\newblock Laplacian simplices.
\newblock {\em Advances in Applied Mathematics}, 114, 2017.

\bibitem{BRENTI2009545}
F.~Brenti and V.~Welker.
\newblock The veronese construction for formal power series and graded
  algebras.
\newblock {\em Advances in Applied Mathematics}, 42(4):545--556, 2009.

\bibitem{Roemer}
M.~Brun and T.~Römer.
\newblock Subdivisions of toric complexes.
\newblock {\em Journal of Algebraic Combinatorics}, 21, 2004.

\bibitem{BrunsRoemer2007}
W.~Bruns and T.~Römer.
\newblock h-vectors of gorenstein polytopes.
\newblock {\em Journal of Combinatorial Theory, Series A}, 114:65--76, 01 2007.

\bibitem{SEP}
Alessio D'Alì, Martina Juhnke-Kubitzke, Daniel Köhne, and Lorenzo Venturello.
\newblock On the gamma-vector of symmetric edge polytopes.
\newblock {\em Preprint arXiv: https://arxiv.org/abs/2201.09835}, 2022.

\bibitem{triangulationsDeLoera}
J.A. {De Loera}, J.~Rambau, and F.~Santos.
\newblock {\em Triangulations. Structures for algorithms and applications},
  volume~25.
\newblock 2010.

\bibitem{Diestel}
R.~Diestel.
\newblock {\em Graph Theory}, volume 173.
\newblock 2017.

\bibitem{Edelsbrunner}
H.~Edelsbrunner and D.~R. Grayson.
\newblock Edgewise subdivision of a simplex.
\newblock In {\em Proceedings of the {F}ifteenth {A}nnual {S}ymposium on
  {C}omputational {G}eometry ({M}iami {B}each, {FL}, 1999)}, pages 24--30. ACM,
  New York, 1999.

\bibitem{Ehrhart}
E.~Ehrhart.
\newblock Sur les poly\`edres rationnels homoth\'{e}tiques \`a {$n$}
  dimensions.
\newblock {\em C. R. Acad. Sci. Paris}, 254:616--618, 1962.

\bibitem{Goldberg}
T.E. Goldberg.
\newblock Combinatorial laplacians of simplicial complexes.
\newblock {\em A Senior Project submitted to The Division of Natural Science
  and Mathematics of Bard College}, 5 2002.

\bibitem{Grayson}
D.~R. Grayson.
\newblock Exterior power operations on higher {$K$}-theory.
\newblock {\em $K$-Theory}, 3(3):247--260, 1989.

\bibitem{Grunbaum}
B.~Gr{\"u}nbaum.
\newblock {Convex Polytopes}.
\newblock {\em Graduate Texts in Mathematics}, 2003.

\bibitem{UnimodularTriangulations}
C.~Haase, A.~Paffenholz, L.~Piechnik, and F.~Santos.
\newblock Existence of unimodular triangulations - positive results.
\newblock {\em Memoirs of the American Mathematical Society}, 270, 05 2014.

\bibitem{HerzogHibiOhsugi2018}
J.~Herzog, T.~Hibi, and H.~Ohsugi.
\newblock {\em Edge Polytopes and Edge Rings}, pages 117--140.
\newblock 09 2018.

\bibitem{HibiReflexive}
T.~Hibi.
\newblock Dual polytopes of rational convex polytopes.
\newblock {\em Combinatorica}, 2(2):237--240, 1992.

\bibitem{HigashitaniJochemkoMichalek2019}
A.~Higashitani, K.~Jochemko, and M.~Micha{\l}ek.
\newblock Arithmetic aspects of symmetric edge polytopes.
\newblock {\em Mathematika}, 65:763--784, 05 2019.

\bibitem{VeroneseConstruction}
K.~Jochemko.
\newblock On the real-rootedness of the veronese construction for rational
  formal power series.
\newblock {\em International Mathematics Research Notices}, 2018:4780--4798,
  2018.

\bibitem{lovaszplummer1986}
L.~Lov\'asz and M.~D. Plummer.
\newblock Matching theory.
\newblock {\em Annals of Discrete Mathematics}, 29, 1986.

\bibitem{meyer2018laplacian}
M.~Meyer and Pllaha T.
\newblock Laplacian simplices ii: A coding theoretic approach, 2018.

\bibitem{Mulas}
R.~Mulas, D.~Horak, and J.~Jost.
\newblock {\em Graphs, Simplicial Complexes and Hypergraphs: Spectral Theory
  and Topology}, pages 1--58.
\newblock Springer International Publishing, Cham, 2022.

\bibitem{Munkers84}
J.R. Munkres.
\newblock {\em {Elements of Algebraic Topology}}.
\newblock Addison Wesley Publishing Company, 1984.

\bibitem{HO}
H.~Ohsugi and T.~Hibi.
\newblock Special simplices and {G}orenstein toric rings.
\newblock {\em J. Combin. Theory Ser. A}, 113(4):718--725, 2006.

\bibitem{OhsugiTsuchiya2021}
H.~Ohsugi and A.~Tsuchiya.
\newblock Pq-type adjacency polytopes of join graphs.
\newblock 03 2021.

\bibitem{StanleyLatticePolytopes}
R.P. Stanley.
\newblock Decompositions of rational convex polytopes.
\newblock {\em Annals of Discrete Math.}, 6:333--342, 1980.

\bibitem{Ziegler}
G.~Ziegler.
\newblock {Lectures on Polytopes}.
\newblock {\em Graduate Texts in Mathematics}, 1995.

\end{thebibliography}

\end{document}